\documentclass[letter,12pt]{article}
\usepackage{amssymb}
\usepackage{amsthm}
\usepackage{amsfonts}
\usepackage{amsmath}
\usepackage{anysize}
\usepackage{bbm}
\usepackage{abstract}
\marginsize{2.5cm}{2.5cm}{0.9cm}{1.9cm}
\usepackage{xcolor}
\usepackage{color}

\newtheorem{theorem}{Theorem}[section]
\newtheorem{lemma}[theorem]{Lemma}
\newtheorem{prop}[theorem]{Proposition}
\newtheorem{coro}[theorem]{Corollary}

\theoremstyle{definition}
\newtheorem{definition}[theorem]{Definition}

\newcommand{\Ec}{\mathcal{E}}

\newcommand{\setdef}{\ \vert \ }

\newcommand{\psh}{{\rm PSH}}

\newcommand{\vol}{{\rm Vol}}

\newcommand{\ddbar}{\partial\bar\partial}

\newcommand{\AM}{{\rm I}}

\newcommand{\PSH}{{\rm PSH}}
\newcommand{\Amp}{{\rm Amp}}
\newcommand{\id}{\mathbbm{1}}

\usepackage{todonotes}

\usepackage{hyperref}
\hypersetup{
    bookmarks=true,         
    unicode=false,          
    pdftoolbar=true,       
    pdfmenubar=true,       
    pdffitwindow=false,    
    pdfstartview={FitH},      
    pdfauthor={T. Darvas, E. Di Nezza, C.H. Lu},     
    colorlinks=true,       % false: boxed links; true: colored links
   linkcolor=black,          % color of internal links
    citecolor=black,        % color of links to bibliography
    filecolor=black,      % color of file links
    urlcolor=black}           % color of external links

\begin{document}

\title{$L^1$ metric geometry of big cohomology classes}

\author{Tam\'as Darvas, Eleonora Di Nezza, Chinh H. Lu}
\date{\emph{\footnotesize{Dedicated to Jean-Pierre Demailly on the occasion of his 60th birthday.}}}

\maketitle

\begin{abstract} Suppose $(X,\omega)$ is a compact K\"ahler manifold of dimension $n$, and $\theta$ is closed $(1,1)$-form representing a big cohomology class. We introduce a metric $d_1$ on the finite energy space $\mathcal{E}^1(X,\theta)$, making it a complete geodesic metric space. This construction is potentially more rigid compared to its analog from the K\"ahler case, as it only relies on pluripotential theory, with no reference to infinite dimensional $L^1$ Finsler geometry. Lastly, by adapting the results of Ross and Witt Nystr\"om to the big case, we show that one can construct geodesic rays in this space in a flexible manner. 
\end{abstract}

\section{Introduction}

Let $(X,\omega)$ be a K\"ahler manifold of complex dimension $n$. Going back to Yau's solution of the Calabi conjecture \cite{Yau78}, the study of complex Monge-Amp\`ere equations on $X$ has received a lot of attention. Several problems in K\"ahler geometry, related to canonical metrics,  boil down to solving an equation of complex Monge-Amp\`ere type. When trying to find weak solutions for such equations, one is naturally led to the space $\mathcal E^1(X,\omega)$, introduced by Guedj and Zeriahi \cite{GZ07} building on previous constructions of Cegrell in the local case \cite{Ce98}. Later, in the work of the first named author \cite{Dar14,Dar15} it was discovered that $\mathcal E^1(X,\omega)$ has a natural metric geometry, arising as the completion of a certain $L^1$ Finsler metric on the space of smooth K\"ahler potentials, an open subset of $C^\infty(X)$, reminiscent of the $L^2$ Riemannian metric of Mabuchi-Semmes-Donaldson (\cite{Ma87,Se92,Do99}). The exploration of the space $\Ec^1(X,\omega)$ and its metric structure led to numerous applications concerning existence of K\"ahler-Einstein and constant scalar curvature metrics (see \cite{BBEGZ11,BBGZ13,BBJ15,BDL16,CC18, DR17,Dar16,DNG16} as well as references in the recent survey \cite{Dar17}). 

For the rest of the paper we consider $\theta$, a closed $(1,1)$-form representing a big cohomology class. As pointed out in \cite{BEGZ10}, one can still consider the space $\mathcal E^1(X,\theta)$, in hopes of finding weak solutions to equations of  complex Monge-Amp\`ere type in the big context. However we can not recover this space using infinite dimensional $L^1$ Finsler geometry, as there is no Fr\'echet manifold readily available in this setting to replace the role of the space of K\"ahler potentials. In our first main result we show that this difficulty can be overcome, by defining the metric structure of $\mathcal E^1(X,\theta)$ directly, using only pluripotential theory, bypassing the Finsler geometry (compare with \cite{Dar16, DNG16} that deal with an intermediate particular case). 
The resulting space still enjoys the same properties as its analog the K\"ahler case, and we expect our construction to have applications in the study of complex Monge-Amp\`ere equations in the context of big cohomology classes. 

Let us briefly introduce the main terminology and concepts, leaving the details to the preliminaries section and thereafter. Roughly speaking, $\mathcal E^1(X,\theta)\subset \textup{PSH}(X,\theta)$ is the set of potentials whose Monge-Amp\`ere energy $I$ is finite. Given $u,v \in \mathcal E^1(X,\theta)$, it has been shown in \cite[Theorem 2.10]{DDL16} that $P(u,v)$, the largest $\theta$-psh function lying below $\min(u,v)$, belongs to $\Ec^1(X,\theta)$. Consequently, we can define $d_1(u,v)$ as the following finite quantity:
\begin{equation}\label{eq: d1_intr_def}
d_1(u,v)=\AM(u) + \AM(v) - 2\AM(P(u,v)).
\end{equation}
Thus defined, $d_1$ is symmetric, and non-degeneracy is a simple consequence of the domination principle. The main difficulty is to show that the triangle inequality also holds. We accomplish this, and we are also able to show that the resulting metric space $(\mathcal E^1(X,\theta),d_1)$ is complete, with metric geodesics running between any two points. These geodesic segments will be constructed as a Perron envelope, generalizing an observation of Berndtsson from the K\"ahler case \cite[Section 2]{Bern}. We record all of this in our first main theorem:

\begin{theorem} $(\mathcal E^1(X,\theta),d_1)$ is a complete geodesic metric space.
\end{theorem}

As alluded to above, in the K\"ahler case the $d_1$ metric is introduced quite differently. In that case, one puts an $L^1$ Finlser metric on the Fr\'echet manifold of smooth K\"ahler potentials and the completion of its path length metric will coincide with $(\mathcal E^1(X,\theta),d_1)$ \cite{Dar15}. In the K\"ahler case, formula \eqref{eq: d1_intr_def} is a result of a theorem (\cite[Corollary 4.14]{Dar15}), but in the big case we take it as our definition for $d_1$!

Though there is no apparent connection with infinite dimensional $L^1$ Finsler geometry in the big case. By the double estimate below, we will still refer to $d_1$ as the $L^1$ metric of $\mathcal E^1(X,\theta)$. Indeed, by this double inequality, it seems that one should think of $d_1$ as a kind of $L^1$ metric with ``moving measures'':
$$d_1(u,v) \leq \int_X |u-v| \theta_u^n + \int_X |u-v| \theta_v^n \leq 3 \cdot 2^{n} (n+1) d_1(u,v), \ \ \ \ u,v \in \mathcal E^1(X,\theta).$$
As a consequence of this inequality, the expression in the middle satisfies a quasi-triangle inequality. We note that this is true for even more general such expressions, as recently proved using completely different methods in \cite{GLZ17}.

To motivate our second main result, a short review of  historical developments is in order. The study of the geometry of the space of K\"ahler metrics was (and still is) closely connected with the uniqueness and existence of canonical K\"ahler metrics. 
The compactness principle of $\mathcal E^1(X,\theta)$ from \cite{BBEGZ11}, and the exploration of the $d_1$-geometry of this space \cite{Dar15} led to many recent advances in this direction \cite{BBJ15,DR17,BDL16}. Going back to earlier developments, Donaldson conjectured that a constant scalar curvature metric exists in a K\"ahler class if and only if the K-energy has certain growth along the geodesic rays of this space \cite{Do99}. This is closely related to the notion of K-stability and is the focus of intense research to this day. Motivated by this picture, there is special interest in regularity of geodesic segments and rays, as well as their geometric significance (see \cite{Ch00, AT03, PS10, CT08, CS14, RWN} to name only very few works in a fast expanding literature). Following the appearance of \cite{Dar15}, it became apparent that a weak version of Donaldson's conjectural picture generalizes to the $d_1$-metric completion. Using a mixture of novel PDE techniques and the method of \cite{DH17}, this latter conjecture was very recently fully addressed by Chen and Cheng \cite{CC18}.  We expect that results in the above papers will eventually find generalizations to the big setting. 

Given their importance in the above mentioned applications, we are interested to see how one can construct weak geodesic rays inside $(\mathcal E^1(X,\theta),d_1)$, with the hopes of using them in later investigations involving big cohomology classes. To this end, we point out below that the construction of Ross and Witt Nystr\"om \cite{RWN} not only generalizes to the big case, but it can be shown that their very flexible method gives all possible weak geodesic rays (with minimal singularity) in a unique manner. 

We skim over the main aspects of the construction.
Suppose $\phi \in \psh(X,\theta)$ has minimal singularity. Roughly speaking, we say that $\Bbb R \ni \tau \to \psi_\tau \in \textup{PSH}(X,\theta)$ is a test curve, if it is $\tau$-concave,  $\psi_\tau = \phi \in \textup{PSH}(X,\theta)$ for all $\tau \leq - C_\psi$, and  $\psi_\tau = -\infty$ for all $\tau \geq C_\psi$, for some constant $C_\psi >0$. Additionally a test curve is maximal, if (using the notation of \cite[Section 1]{DDL17}, see \eqref{eq: P_sing_def} below) it satisfies:
$$P[\psi_\tau](\phi)=\psi_\tau, \ \ \tau \in \Bbb R.$$
As opposed to weak geodesic rays, test curves can be easily constructed, and they can also be conveniently maximized (Proposition \ref{prop: maximization}). Roughly speaking,  our second main result points out a duality between rays and maximal test curves, via the partial Legendre transform:

\begin{theorem} The correspondence $\psi \to \check \psi$ gives a bijective map between maximal $\tau$-usc test curves $\tau \to \psi_\tau$ and weak geodesic rays with minimal singularity type $t \to u_t$. The inverse of this map is $u \to \hat u$.
\end{theorem}
Here $\check \psi$ and $\hat u$ represent the partial (inverse) Legendre transforms of $\psi$ and $u$ respectively, defined by:
$$\check \psi_t := \sup_{\tau \in \Bbb R}(u_\tau + t\tau), \ \ \ \hat u_\tau := \inf_{t \geq 0} (u_t - t\tau).$$
As a corollary we recover the main analytic result of \cite{RWN} in the big context:
\begin{coro}\label{cor: main_from test curve to ray} Let $\tau\rightarrow \psi_\tau$ be a test curve such that $\psi_{-\infty} =\phi$.
Define $$w_t= \sup_{\tau\in \mathbb{R}} (P[\psi_\tau](\phi) + t\tau), \ \ t \geq 0.$$ Then the curve $t\rightarrow w_t$ is a weak geodesic ray, with minimal singularity, emanating from $\phi$.
\end{coro}

\paragraph{Organization.} 
Our notation and terminology carry over from \cite{DDL16} and \cite{DDL17}. In Section \ref{preliminaries} we review some background on the Monge-Amp\`ere theory in big cohomology classes. Section \ref{section:  distance d1} is devoted to the proof of Theorem 1.1. In Section \ref{section: rays} we adapt the concepts of \cite{RWN} to the big context and prove Theorem 1.2 and Corollary 1.3.
 \medskip
 
 \paragraph{Acknowledgement.} We thank the referee for careful reading and useful comments that significantly improved the presentation.

\section{Preliminaries}\label{preliminaries}
We lay down our notation, and we review several basic results from the Monge-Amp\`ere theory and geodesics in big cohomology classes.
Let $X$ be a compact K\"ahler manifold of dimension $n$ and fix $\theta$, a closed smooth real $(1,1)$-form on $X$.

A potential $u\in L^1(X,\omega^n)$  is quasi-plurisubharmonic (quasi-psh for short) if near every point $x \in X$ there exists a coordinate patch $V \subset X$, identifying $x \in X$ with $0 \in \Bbb C^n$, such that  $u|_V$ is the difference of a plurisubharmonic (psh)  and a smooth function. Additionally, $u$ is called $\theta$-plurisubharmonic ($\theta$-psh for short) if $\theta_u :=\theta+i\ddbar u \geq 0$ in the sense of currents. The set of all $\theta$-psh functions  is denoted by $\textup{PSH}(X,\theta)$. In our convention, the potential equal to $-\infty$ everywhere on $X$ is an element of $\textup{PSH}(X,\theta)$.

We say that $\{\theta \}$ is \emph{pseudoeffective} if $\textup{PSH}(X,\theta)$ is non-empty. Along these lines, $\{\theta\}$ is \emph{big} if $\textup{PSH}(X,\theta-\varepsilon \omega)$ is non-empty for some $\varepsilon >0$.  

In case $\{\theta\}$ is big, the \emph{ample locus} $\textup{Amp}(\{\theta\}) \subset X$ is the open dense set of points $x \in X$ such that there exists $u \in \textup{PSH}(X,\omega)$, smooth in a neigborhood of $x$, and satisfying $\theta + i\ddbar u > \varepsilon \omega$ in the same neighborhood, for some $\varepsilon(x)>0$. 

Let $u \in \textup{PSH}(X,\theta)$. Given that locally $u$ can be written as a sum of a psh function and a smooth function we obtain that 
$$u(x) = \lim_{r \to 0}\frac{1}{d\mu(B(x,r))} \int_{B(x,r)}u(y) d\mu(y), \ \ x \in X,$$
where $B(x,r)$ is a coordinate ball of radius $r>0$ centered at $x \in X$, and $d\mu$ is the Lebesque measure (see \cite[Theorem 1.2.3(iv)]{Bl97}). As a consequence of this we immediately obtain the following:
\begin{lemma}\label{lem: aeineq_everywhere} Let $u, v \in \textup{PSH}(X,\theta)$ such that $u \leq v$ a.e. on $X$. Then $u \leq v$ everywhere on $X$.
\end{lemma}

Given $u_1,\ldots,u_p \in \textup{PSH}(X,\theta)$, we recall the definition of the current $\theta_{u_1}\wedge \ldots \wedge \theta_{u_p}$ from \cite[Section 1]{BEGZ10}.  This generalizes a construction of Bedford and Taylor \cite{BT76, BT87} applicable for bounded potentials. Indeed, in a small enough coordinate patch $V \subset X$ we can write that $\theta := i\partial \bar \partial \phi$ for some $\phi \in C^\infty(X)$. Then in this neighborhood we introduce:
\begin{equation}\label{eq: non-pluripolar_def}
\theta_{u_1}\wedge \ldots \wedge \theta_{u_p}|_V:= \lim_{k \to \infty} \mathbbm{1}_{V \cap \{\phi + u_1 > -k\} \cap \ldots\cap \{\phi + u_p > -k\}}\theta_{\max(\phi + u_1, - k)}\wedge \ldots \wedge \theta_{\max(\phi + u_p,-k)},
\end{equation}

In \cite[Section 1]{BEGZ10} it is  argued that this limit of currents is well-defined, invariant under change of coordinates  and it is a  closed positive $(p,p)$-current which does not charge pluripolar sets. For a $\theta$-psh function $u$, the \emph{non-pluripolar complex Monge-Amp{\`e}re measure} of $u$ is simply $\theta_u^n:=\theta_u\wedge \ldots\wedge \theta_u$.  

Many properties of  $\theta_{u_1}\wedge \ldots \wedge \theta_{u_p}$ carry over from Bedford--Taylor theory \cite{BT76, BT87}  directly. Here we only highlight the ones that will come up the most in this work, for example locality of \eqref{eq: non-pluripolar_def} with respect to the plurfine topology. This latter topology is the coarsest topology making local plurisubharmonic functions continuous on $X$, and it is easy to see that it refines the usual Euclidean topology. Moreover, from \cite[Proposition 1.4]{BEGZ10} it follows that the construction  \eqref{eq: non-pluripolar_def} is \emph{local in the plurifine topology}:
\begin{lemma}\label{lem: plurifine_prop} If $u_j,v_j \in \textup{PSH}(X,\theta)$ such that 
and $u_j = v_j$ on a plurifine open set $O \subset X$. Then
$$\mathbbm{1}_{O} \theta_{u_1}\wedge \ldots \wedge \theta_{u_p}=\mathbbm{1}_{O}\theta_{v_1}\wedge \ldots \wedge \theta_{v_p}.$$
\end{lemma}

Moreover, as pointed out by Guedj-Zeriahi \cite[Corollary 2.8]{GZ05}, every element $u \in \textup{PSH}(X,\theta)$ is \emph{quasi-continuous} in the sense that for any $\varepsilon>0$ it is possible to find a Euclidean open set $O \subset X$ such that $u|_{O}$ is continuous and $\textup{Cap}_\theta(X\setminus O) \leq \varepsilon$. By $\textup{Cap}_\theta(\cdot)$ we mean the Monge--Amp\`ere capacity defined in \cite[Section 4.1]{BEGZ10}. We note that by \cite[Theorem 2.8]{DDL16} all notions of Monge--Amp\`ere capacity are (essentially) independent of the choice of form $\theta$. 

Related to the above, we say that a sequence  of functions $\{f_j\}_j$ \emph{converges in capacity} to a function $f$ on $X$ if $\lim_{\varepsilon \to 0} \textup{Cap}_\theta\{|f_j -f| > \varepsilon\} =0$. When $f_j,f$ are $\theta$-psh then convergence in capacity has important ramifications related to convergence of non-pluripolar measures (see \cite[Theorem 2.3]{DDL17}).

For an extensive treatment of non-pluripolar products in the setting of big cohomology classes we refer to \cite[Section 1 and  2]{BEGZ10} and \cite[Section 2 and 3]{DDL17}.

If $u, v \in \textup{PSH}(X,\theta)$, then $u$ is said to be \emph{less singular} than $v$ if $v\leq u+C$ for some $C\in \Bbb R$, while they are said to have the \emph{same singularity type} if $u-C \leq v\leq u+C$, for some $C\in \mathbb{R}$. A  $\theta$-psh function $u$ is said to have \emph{minimal singularity type} if it is less singular than any other $\theta$-psh function. An example of a $\theta$-psh function with minimal singularity is 
$$
V_\theta(x):=\sup\{ u(x) \setdef  u \in \psh(X, \theta),\; \; u\leq 0\}.
$$
For simplicity, in the whole paper we make the following normalization: 
$$\vol(\{\theta\}):=\int_X \theta_{V_{\theta}}^n =1.$$ 
By multiplying $\theta$ with a constant this can always be attained.
\begin{lemma}
	\label{lem: Berman Vtheta}
	We have $\theta_{V_{\theta}}^n \leq \id_{\{V_{\theta}=0\}} \theta^n$. 
\end{lemma}
Note here that the form $\theta$ may fail to be semi-positive at some points but it is semipositive on the set $\{V_{\theta}=0\}$. This lemma is a result of Berman \cite{Ber13} (for a detailed argument in the big case we refer to \cite[Theorem 2.6]{DDL16} (arXiv version)). 

If $u$ has minimal singularity type  then $\int_X \theta_u^n$, the \emph{total mass} of $\theta_u^n$, is equal to $\int_X \theta_{V_{\theta}}^n$ which was normalized to be $1$. With this convention, for a general $u\in \PSH(X,\theta)$,  $\int_X \theta_u^n$ may take any value in $[0,1]$. Lastly, according to \cite[Theorem 1.2]{WN17} if $u$ is less singular than $v$ then $\int_X \theta_v^n \leq \int_X \theta_u^n$.

\subsection{The energy functionals}

If $u\in \PSH(X,\theta)$ has minimal singularity type then its Monge-Amp\`ere energy  is defined as
\[
\mathrm{I} (u) :=\frac{1}{n+1} \sum_{k=0}^n \int_X (u-V_\theta) \theta_u^k \wedge\theta_{V_\theta}^{n-k}. 
\]

\noindent We collect basic properties of the Monge-Amp\`ere energy:

\begin{theorem}\label{thm: basic I energy} Suppose $u,v \in \PSH(X, \theta)$ have minimal singularity type. The following hold:\\
\noindent (i) $ \mathrm{I}(u)-\mathrm{I}(v) = \frac{1}{n+1}\sum_{k=0}^n \int_X (u-v) \theta_{u}^k \wedge \theta_{v}^{n-k}.$\\
\noindent (ii) $\mathrm{I}$ is non-decreasing and concave along affine curves. Additionally, the following estimates hold: $
	\int_X (u-v) \theta_u^n \leq \AM(u) -\AM(v) \leq \int_X (u-v) \theta_v^n.$\\
\noindent (iii) If $v\leq u$ then, $
\frac{1}{n+1} \int_X (u-v) \theta_v^n \leq  \AM(u) - \AM(v) \leq  \int_X (u-v) \theta_{v}^n. $ In particular, $\AM(v) \leq \AM(u)$
\end{theorem}

In the K\"ahler case, the above formulas and inequalities can be established using integration by parts. When dealing with potentials  having minimal singularity, integration by parts works in the big case as well \cite[Theorem 1.14]{BEGZ10}, hence the proof from the K\"ahler case works with only superficial changes (see \cite[Section 2.2]{BEGZ10}).

Using the monotonicity property of $I=$ from above we can introduce the Monge-Amp\`ere energy for arbitrary $u\in \mathrm{PSH}(X,\theta)$ as
\[
\AM(u) : =\inf \{\AM(v) \setdef  v\in \mathrm{PSH} (X,\theta), \; v\ \textrm{has minimal singularity type, and } u\leq v\}. 
\]
We let $\mathcal{E}^1(X,\theta)$ denote the set of all $u\in \psh(X,\theta)$  such that $\AM(u)>-\infty$. Since $\theta$ will be fixed throughout the paper we will often denote this space simply as $\mathcal{E}^1$.  As shown in \cite[Proposition 2.10]{BEGZ10} the functional $\AM$, thus defined on $\PSH(X,\theta)$ (and may take value $-\infty$) is non-decreasing, concave, upper semicontinuous on ${\rm PSH}(X,\theta)$, and continuous along decreasing sequences. 

It follows from \cite[Proposition 2.11]{BEGZ10} that $\int_X (V_{\theta}-u) \theta_u^n <+\infty$ whenever $u\in \Ec^1$. For $C>0$, by Lemma \ref{lem: plurifine_prop} we have $\id_{\{u>V_{\theta}-C\}} \theta_{\max(u,V_{\theta}-C)}^n= \id_{\{u>V_{\theta}-C\}} \theta_{u}^n$.  Since $\int_X \theta_{u}^n=\int_X \theta_{\max(u,V_{\theta}-C)}^n=1$ we can write
\begin{eqnarray}
\label{eq: E1 by mass}
\lim_{C\to +\infty} C\int_{\{u\leq V_{\theta}-C\}}\theta_{\max(u,V_{\theta}-C)}^n & = & \lim_{C\to +\infty} C\int_{\{u\leq V_{\theta}-C\}}\theta_{u}^n  \\
&\leq & \lim_{C\to +\infty} \int_{\{u\leq V_{\theta}-C\}} (V_{\theta}-u)\theta_{u}^n =0. \nonumber
\end{eqnarray}

\begin{prop}\label{prop: basic I energy}
	The conclusions of Theorem \ref{thm: basic I energy} still hold for $u,v \in \Ec^1$. 
\end{prop}
\begin{proof}
	We can assume that $u,v\leq 0$.  We set $u^C:= \max(u,V_{\theta}-C)$ for $C>0$. We want to prove that, for $k\in \{0,...,n\}$, 
    \begin{equation}\label{eq: basic I energy 11}
		\lim_{C\to +\infty}\int_X (u^C-v^C) \theta_{u^C}^k \wedge \theta_{v^C}^{n-k}  =\int_X (u-v) \theta_{u}^k \wedge \theta_{v}^{n-k}. 
	\end{equation}
    Clearly, it suffices to check that
	\begin{equation}\label{eq: basic I energy 1}
		\lim_{C\to +\infty}\int_X (u^C-V_{\theta}) \theta_{u^C}^k \wedge \theta_{v^C}^{n-k}  =\int_X (u-V_{\theta}) \theta_{u}^k \wedge \theta_{v}^{n-k}. 
	\end{equation}
	By decomposing the integral into two parts $\int_{\{\min(u,v)>V_{\theta}-C\}}$ and $\int_{\{\min(u,v)\leq V_{\theta}-C\}}$, using Lemma \ref{lem: plurifine_prop} and noting that $\{\min(u,v)\leq V_{\theta}-C\} \subseteq \{u\leq V_{\theta}-C\}\cup \{v\leq V_{\theta}-C\}$, we see that proving \eqref{eq: basic I energy 1}  boils down to showing that 
	\begin{equation}
		\label{eq: energy estimate 1}
		\lim_{C\to +\infty} C \int_{\{u\leq V_{\theta}-C\}}\theta_{u^C}^k \wedge \theta_{v^C}^{n-k} =0,\  \text{and}\ \lim_{C\to +\infty} C \int_{\{v\leq V_{\theta}-C\}}\theta_{u^C}^k \wedge \theta_{v^C}^{n-k} =0, \  \forall k. 
	\end{equation}

	We will prove the first equality and the same arguments apply to prove the second one. 
	Observing that $V_\theta-C \leq v^C \leq V_\theta$ we have the inclusion 
	\[
	\{u\leq V_{\theta}-C\} \subset \left \{u^C\leq \frac{v^C+V_{\theta}-C}{2}\right \} \subset  \{u\leq V_{\theta}-C/2\}.
	\]
	Using the partial comparison principle \cite[Proposition 2.2]{BEGZ10} and that 
    $$\theta_{v^C}^{n-k} \leq  2^{n-k} \theta_{\frac{v^C+V_\theta-C}{2}}^{n-k}$$ 
    we get 
	\begin{eqnarray*}
		C \int_{\{u\leq V_{\theta}-C\}}\theta_{u^C}^k \wedge \theta_{v^C}^{n-k}  &\leq &  C   \int_{\{u^C\leq \frac{v^C+V_{\theta}-C}{2}\}}\theta_{u^C}^k \wedge \theta_{v^C}^{n-k}\\
		&\leq & 2^{n-k} C  \int_{\{u^C\leq \frac{v^C+V_{\theta}-C}{2}\}}\theta_{u^C}^n\\
		&\leq & 2^{n-k} C  \int_{\{u\leq V_{\theta}-C/2\}}\theta_{u^C}^n\\
		&\leq & 2^{n-k} C \int_{\{u\leq V_{\theta}-C/2\}}\theta_{u}^n,
	\end{eqnarray*}
    where in the last inequality we used Lemma \ref{lem: plurifine_prop} and $1 = \int_X \theta_{u}^n= \int_X \theta_{u^C}^n$.
From this and \eqref{eq: E1 by mass} we obtain \eqref{eq: energy estimate 1}, hence  \eqref{eq: basic I energy 1}, completing the proof. 
\end{proof}

\begin{lemma}\label{lem: bound energy 1}
	If $u\leq v \leq 0$ are in $\Ec^1$ then, for every $C>0$,
	\[
	 \theta_v^n (v\leq V_{\theta}-C) \leq 2^{n}  \theta_u^n(u\leq V_{\theta}-C/2). 
 	\]
\end{lemma}
\begin{proof}
Fix  $C>0$ and set $w:= \frac{v+V_{\theta}-C}{2}$. Using the inclusion of sets 
\[
\{v\leq V_{\theta}-C\} \subset \left \{u\leq w\right \} \subset  \{u\leq V_{\theta}-C/2\}
\]
and  the comparison principle  \cite[Corollary 2.3]{BEGZ10} we obtain 
\begin{eqnarray*}
	 \theta_v^n (v\leq V_{\theta}-C)  &\leq &    \theta_v^n(u\leq w) 
\leq    2^n \theta_{w}^n(u\leq w) \\ 
	&\leq &    2^n  \theta_u^n(u\leq w) \leq    2^n   \theta_u^n(u\leq V_{\theta}-C/2). 
\end{eqnarray*}
\end{proof}

In the study of the metric space $(\mathcal{E}^1,d_1)$ we will also make use of the $I_1$-functional introduced in \cite{Dar15} (inspired by the $I_2$ functional of \cite{G14}): 
\[
I_1(u,v) =\int_X |u-v| (\theta_u^n+\theta_v^n), \quad u,v\in \Ec^1(X,\theta).
\]
It follows directly from Lemma \ref{lem: plurifine_prop} that 
\begin{equation}
	\label{eq: plurifine Pythagore max}
	I_1(u,v) = I_1(\max(u,v),u) + I_1(\max(u,v),v), \ \forall u,v \in \Ec^1. 
\end{equation}

\begin{prop}\label{prop: BEGZ convergence I and I1}
	Let $\{u_j\}_j \subset \mathcal{E}^1$ be  a sequence  converging decreasingly (or increasingly a.e.) towards $u\in \mathcal{E}^1$.  Then $I_1(u_j,u) \to 0$ and $\AM(u_j) \to \AM(u)$. 
\end{prop}
\begin{proof}
Observe first that in the case the functions $u_j, u$ have minimal singularity type the result was known (see e.g. \cite[Proposition 2.10, Theorem 2.17]{BEGZ10}, \cite[Proposition 4.3]{BB10}, or \cite[Lemma 4.1]{DDL17}).

We first prove the convergence of $\AM$. 
If the sequence is decreasing this was known by \cite[Proposition 2.10]{BEGZ10}. Assume now that $u_j\nearrow u\leq 0$. Again we denote $u^C:= \max(u,V_{\theta}-C)$ and observe that $u^C$ and $u_j^C$ have minimal singularity type. Since $\AM(u_j^C) \to \AM(u^C)$ as $j\to +\infty$ for any $C>0$ fixed and $\AM(u^C) \to \AM(u)$ as $C\to +\infty$, it suffices to show that 
\[
\lim_{C\to +\infty} (\AM(u_j^C) -\AM(u_j))  =0
\]
uniformly in $j$. By concavity (Proposition \ref{prop: basic I energy}) we have that 
\begin{eqnarray*}
	0\leq \AM(u_j^C) -\AM(u_j)& \leq & \int_X (u_j^C-u_j) \theta_{u_j}^n
	 \leq   \int_{\{u_j\leq V_{\theta}-C\}} (V_{\theta}-C-u_j) \theta_{u_j}^n\\
	 & = &  \int_{C}^{+\infty} \theta_{u_j}^n (u_j\leq V_{\theta}-t) dt.
\end{eqnarray*}
But it follows from Lemma \ref{lem: bound energy 1} that 
\[
\int_{\{u_j\leq V_{\theta}-t\}} \theta_{u_j}^n \leq 2^{n} \int_{\{u_1\leq V_{\theta}-t/2\}} \theta_{u_1}^n.
\]
Hence we can continue the above estimate and write
\begin{eqnarray*}
	0\leq \AM(u_j^C) -\AM(u_j)	 & \leq  &  \int_{C}^{+\infty} \theta_{u_j}^n (u_j\leq V_{\theta}-t) dt\\
	&\leq & 2^{n+1}\int_{C/2}^{+\infty}  \theta_{u_1}^n(u_1\leq V_{\theta}-t)dt\\
	&=& 2^{n+1} \int_{\{u_1\leq V_{\theta}-C/2\}} (V_{\theta}-u_1-C/2)  \theta_{u_1}^n. 
\end{eqnarray*}
 Since $u_1\in \Ec^1$ the last term above converges to $0$ as $C\to +\infty$ (in view of \eqref{eq: E1 by mass}), finishing the proof of the convergence of $\AM$.

We now prove the convergence of $I_1$. It follows from Proposition \ref{prop: basic I energy}  that $\AM(u_j)-\AM(u)$ is the sum of $(n+1)$ terms having the same sign (which is positive if the sequence is decreasing and negative if the sequence is increasing). Hence the convergence of $\AM$ implies that each term converges to $0$. In particular, 
\[
\lim_{j\to +\infty}\int_X |u_j-u|(\theta_u^n +\theta_{u_j}^n)  = \lim_{j\to +\infty}\int_X (u_j-u)(\theta_u^n +\theta_{u_j}^n)  =0. 
\]
\end{proof}

Next we record a particular case of the domination principle (see \cite[Proposition 5.9]{BL12} and \cite[Proposition 2.4]{DDL16}) that will be useful for us:

\begin{prop}\label{prop: domination} Let $u,v \in \mathcal E^1$ such that $u \leq v$ a.e. with respect to $\theta_v^n$. Then $u \leq v$.
\end{prop}

The next result is a consequence of \cite[Lemma 5.8]{BBGZ13} and its proof:

\begin{prop}\label{prop: I_est} Suppose $C>0$ and $\phi,\psi,u,v \in \mathcal E^1$ satisfies
\[
I_1(\phi,V_{\theta}),I_1(\psi,V_{\theta}),I_1(u,V_{\theta}),I_1(v,V_{\theta}) \leq C.
\]

Then there exists a continuous increasing function $f_C: \Bbb R^+ \to \Bbb R^+ $ (only dependent on $C$) with $f_C(0)=0$ such that
\begin{equation}\label{eq: I_script_est2}\Big|\int_X (u-v) (\theta_\phi^n-\theta_\psi^n)\Big| \leq f_C(I_1(u,v)).
\end{equation}
\end{prop}

Following the terminology and results of  \cite{BEGZ10,Dar15} we say that a sequence $\{u_j\}_j\subset \mathcal{E}^1$ converges in energy towards $u\in \mathcal{E}^1$ if $I_1(u_j,u)\to 0$ as $j\to +\infty$.

\subsection{Quasi-psh envelopes}\label{section: envelope}
Given a measurable function $f$ on $X$ we define 
\[
P(f) := P_{\theta}(f) := \textup{usc}\left (\sup \{u \in \PSH(X,\theta) \setdef u \leq f \}\right),
\]
as the largest $\theta$-psh function lying below $f$. If $f=\min(u,v)$ for $u,v$ quasi-psh then there is no need to take the upper semicontinuous regularization in the definition of $P(u,v):= P_{\theta}(\min(u,v))$. The latter is the largest $\theta$-psh function lying below both $u$ and $v$, and was called the rooftop envelope of $u$ and $v$ in \cite{DR16}.

Given $\phi, \psi \in \PSH(X,\theta)$ the envelope of $\phi$ with respect to the singularity type of $\psi$, introduced by Ross and Witt-Nystr\"om \cite{RWN}, is defined as \begin{equation}\label{eq: P_sing_def}
P[\psi](\phi) := \textup{usc}\Big(\lim_{C \to +\infty}P(\psi+C,\phi)\Big).
\end{equation}
When $\phi = V_\theta$, we will simply write $P[\psi]:=P[\psi](V_\theta)$. This potential is the maximal element of the set of $u\in \PSH(X,\theta), u\leq 0$ and $\int_X \theta_u^n = \int_X \theta_{\psi}^n$ as shown in \cite{DDL17}. 

\begin{lemma}
	\label{lem: mass envelope}
	Suppose $u,v \in \PSH(X,\theta)$and $v$ is less singular than $u$.  Then  
	$$
	\int_X \theta_{P[u](v)}^n = \int_X \theta_{u}^n. 
	$$
\end{lemma}
The proof is essentially given in \cite{DDL17} but we recall it here for the reader's convenience. 
\begin{proof}
Let $C>0$ be such that $v\geq u-C$. For each $j>0$ set $u_j:= P(u+j,v)$. Then $u_j$ has the same singularity type as $u$ since $u-C\leq u_j \leq u+j$. It follows from \cite[Theorem 1.2]{WN17} that $\int_X \theta_{u_j}^n =\int_X \theta_u^n$. By definition $u_j \nearrow P[u](v)$  a.e. on $X$. It thus follows from \cite[Theorem 2.3 and Remark 2.5]{DDL17} that 
	$$
	\int_X \theta_{u}^n=\lim_{j\to +\infty} \int_X \theta_{u_j}^n = \int_X \theta_{P[u](v)}^n. 
	$$
\end{proof}

We refer to \cite{DDL16} and \cite{DDL17} for a detailed account on the properties of such envelopes that goes beyond the scope of our present investigations.

Finally, we recall that the Monge-Amp\`ere measure of such envelopes is concentrated on the contact set. Indeed, thanks to \cite[Lemma 3.7]{DDL17} we know that if $\psi,\phi \in \mathrm{PSH}(X, \theta)$ and $P(\psi,\phi)\neq -\infty$ then
\begin{equation}\label{eq: env_measure_concentrate}
	\theta_{P(\psi, \phi)}^n \leq\mathbbm{1}_{\{P(\psi, \phi)=\psi\}} \theta_{\psi}^n + \mathbbm{1}_{\{P(\psi, \phi)=\phi\}} \theta_{\phi}^n. 
\end{equation}
In the K\"ahler case this was proved in \cite[Proposition 3.3]{Dar14}. Moreover, \cite[Theorem 3.8]{DDL17} ensures that
$$\theta^n_{P[\psi]} \leq \mathbbm{1}_{\{P[\psi]=0\}}\theta^n.$$
In the following  we are going to make use of the above inequalities in a crucial way.

\subsection{Weak geodesic segments and rays} \label{section ray}
In this subsection, following  Berndtsson \cite{Bern} we  adapt the definition of (sub)geodesics to the context of big cohomology classes (see also \cite{DDL16}). 

Fix $0 < \ell \leq \infty$. For a curve $(0,\ell) \ni t \mapsto u_t \in \PSH(X,\theta)$  we define its complexification as a function in $X\times D_{\ell}$,  
\[
X\times D_{\ell} \ni (x,z) \mapsto U(x,z) := u_{\log |z|}(x),
\]
where $D_{\ell}:= \{z\in \mathbb{C} \setdef 1< |z|<e^{\ell}\}$, and $\pi$ is the projection on $X$. 
\begin{definition}
	We say that $t \to u_t$ is a subgeodesic segment (resp. ray) if  $U(x,z) \in \textup{PSH}(X \times D_{\ell},\pi^{*}\theta)$ with $\ell <\infty$ (resp. $U(x,z) \in \textup{PSH}(X \times D_{\infty},\pi^{*}\theta)$). 
\end{definition}

Before proceeding, let us recall the Kiselman minimum principle adapted to our context \cite[Theorem 2.2]{Kis78}:

\begin{theorem}\label{thm: Kiselman} Let $(0,\ell) \ni t \to u_t \in \textup{PSH}(X,\theta)$ be a subgeodesic segment or ray. Given $x \in X$ define $v(x): = \inf_{t \in (0,\ell)} u_t(x)$. Then $v\in \textup{PSH}(X,\theta)$, with  $v$ possible equal to $-\infty$ everywhere.
\end{theorem}
\begin{proof} This is indeed a straightforward consequence of the (local) Kiselman principle, applicable for domains of $\mathbb{C}^m$. A simple proof of the local result can be found in \cite[Theorem I.7.5]{De12}. The general result  follows after an analysis of $U \in \textup{PSH}(X \times D_l,\pi^* \theta)$ in coordinate patches of $X$. 
\end{proof}

\begin{definition}
	For  $\varphi,\psi \in \PSH(X,\theta)$, we let $\mathcal{S}_{(0,\ell)}(\varphi,\psi)$ denote the set of  all subgeodesic segments $(0,\ell) \ni t \mapsto u_t\in \PSH(X,\theta)$ that satisfy $\limsup_{t\to 0} u_t\leq \varphi$ and $\limsup_{t\to \ell} u_t\leq \psi$. 
\end{definition}

Now, for $\varphi,\psi \in \PSH(X,\theta)$,  the \emph{weak (Mabuchi) geodesic segment} connecting $\varphi$ and $\psi$ is defined as the upper envelope of all subgeodesic segments in $\mathcal{S}_{(0,\ell)}(\varphi,\psi)$, i.e.
\begin{equation}\label{eq: weak_geod_def}
\varphi_t := \sup_{\mathcal{S}_{(0,\ell)}(\varphi,\psi)} u_t.
\end{equation}

For general $\varphi,\psi \in \textup{PSH}(X,\theta)$ it is possible that $\varphi_t$ is identically equal to $-\infty$ for any $t \in (0,\ell)$. But in the case when $\varphi, \psi \in \mathcal{E}^1(X,\theta)$, it was shown in \cite[Theorem 2.10]{DDL16} that $P(\varphi,\psi) \in \mathcal{E}^1(X,\theta)$. Since $P(\varphi,\psi) \leq \varphi_t$, we obtain that $\varphi_t \in \mathcal{E}^1(X,\theta)$ for any $t \in [0,\ell]$ \cite[Proposition 2.14]{BEGZ10}. By $\Bbb R$-invariance each subgeodesic segment is in particular $t$-convex, hence we get that 
\begin{equation}\label{eq: conv_upper_bound}
\varphi_t\leq \left (1-\frac{t}{\ell}\right )\varphi + \frac{t}{\ell} \psi, \ \forall t\in [0,\ell]. 
\end{equation}
Consequently the upper semicontinuous regularization (with respect to both variables $x,z$) of $t \to \varphi_t$ is again in $\mathcal{S}_{(0,\ell)}(\varphi,\psi)$, hence so is $t \to \varphi_t$.  

In particular, if $\varphi$ and $\psi$ have minimal singularity type, the function $h:=|\varphi-\psi|$ is bounded and $t \to u_t:=\max\big(\varphi-\|h\|_{L^\infty}\frac{t}{\ell}, \psi-\|h\|_{L^\infty}\frac{\ell-t}{\ell}\big)$ is a subgeodesic. Therefore  $\varphi_t\geq u_t$ for any $t\in(0,\ell)$ and hence $\varphi_t\in \psh(X, \theta)$ has minimal singularity type for any $t\in (0,\ell)$. Moreover, by this last fact and \eqref{eq: conv_upper_bound} it follows that $\lim_{t \to 1} \varphi_t=\varphi$ and $\lim_{t \to \ell} \varphi_t=\psi$. Consequently, in the particular case when $\varphi,\psi$ have minimal singularity type, it is natural to extend the curves $(0,\ell) \ni t \to \varphi_t \in \textup{PSH}(X,\theta)$ at the endpoints by $\varphi_0 :=\varphi$ and $\varphi_1:=\psi$. As we will see, a similar pattern will arise when $\varphi,\psi \in \mathcal E^1(X,\omega)$.

Collecting and expanding some of the above thoughts, we recall the following lemma \cite[Lemma 3.1]{DDL16}:

\begin{lemma}
        \label{lem: boundary limit of weak geodesic}
        Let $t \to \varphi_t$ be the weak Mabuchi geodesic joining $\varphi_0,\varphi_{\ell} \in \PSH(X,\theta)$ with minimal singularity type, constructed as above. Then for $C:= \sup_X |\varphi_\ell - \varphi_0 |/\ell>0$ we have that 
        \begin{equation*}
               % \label{eq: boundary limit of weak geodesic}
                |\varphi_t -\varphi_{t'}| \leq C|t-t'|, \ \  t,t'\in [0,\ell].
        \end{equation*}
Additionally, for the complexification $\Phi(x,z):= \varphi_{\log|z|}(x)$ we have 
\begin{equation*}%\label{eq: geodesicequation}
(\pi^*\theta +i \partial \bar \partial \Phi)^{n+1}=0 \textup{ in}\; \Amp(\{\theta\}) \times D_\ell,
\end{equation*} 
where equality is understood in the weak sense of measures. 
\end{lemma}

Before proceeding we note that due to our ``Perron type" definition of weak geodesic segments \eqref{eq: weak_geod_def} we automatically get the following comparison principle:

\begin{prop}[Comparison principle]\label{prop: DDL_comp_princ_geod} Let $u_0,u_1,v_0,v_1 \in \textup{PSH}(X,\omega)$ such that $v_0 \leq u_0$ and $v_1 \leq u_1$. If $(0,1) \ni t \to u_t \in \textup{PSH}(X,\theta)$ is the weak geodesic connecting $u_0,u_1$ and $(0,1) \ni t \to v_t \in \textup{PSH}(X,\theta)$ is a weak subgeodesic connecting $v_0,v_1$ then $v_t \leq u_t$ for any $t \in [0,1]$.
\end{prop}

Due to Proposition \ref{prop: DDL_comp_princ_geod}, if $(0,\ell)\ni t \rightarrow \varphi_t \in \textup{PSH}(X,\theta)$ is a weak geodesic segment with minimal singularity type and $a,b,c,d\in (0,\ell)$, then exactly the same arguments  as in \cite[Theorem 3.4]{Dar14} give that  
\begin{equation}\label{eq: m_M_normalization}
m_\varphi:=\inf_{\textup{Amp}(\{\theta\})} \frac{\varphi_a-\varphi_b}{a-b}= \inf_{\textup{Amp}(\{\theta\})} \frac{\varphi_c-\varphi_d}{c-d}, \ \  M_\varphi:=\sup_{\textup{Amp}(\{\theta\})} \frac{\varphi_a-\varphi_b}{a-b}= \sup_{\textup{Amp}(\{\theta\})} \frac{\varphi_c-\varphi_d}{c-d}.
\end{equation}

\medskip

A curve $[0, +\infty)\ni t \rightarrow \varphi_t\in \psh(X, \theta)$ is a \emph{weak geodesic ray}, with minimal singularity type, if for any fixed $\ell>0$ $[0, \ell]\ni t\rightarrow \varphi_t \in \psh(X,\theta)$ is a weak geodesic segment joining $\varphi_0$ and $\varphi_\ell$, potentials with minimal singularity.

\section{The metric space $(\mathcal{E}^1,d_1)$}\label{section:  distance d1}

Let $u,v\in \mathcal{E}^1$. It follows from \cite[Theorem 2.10]{DDL16} that $P(u,v)$ belongs to $\Ec^1$.  In this section we will introduce and study the properties of a complete metric structure on $\mathcal E^1$. The metric will be defined by the following expression
\begin{equation}\label{eq: d_1_eq}
d_1(u,v) := \AM(u) + \AM(v) -2 \AM(P(u,v)). 
\end{equation}
Before we prove that this expression does indeed give a metric, we provide some motivation. When $\theta$ is K\"ahler, it is possible to introduce an $L^1$ Finsler structure on the space of smooth K\"ahler potentials, (see \cite[Section 1.1]{Dar15}). As shown in \cite[Corollary 4.14]{Dar15} the path length metric associated to this Finsler structure is given by \eqref{eq: d_1_eq}. We will show below that it is possible to start with  \eqref{eq: d_1_eq} and avoid inifinite dimensional Finsler geometry all together. This line of thought is especially fruitful in the case of big classes, where the space of smooth K\"ahler potentials has no analog to begin with.

\subsection{$d_1$ is a metric}
The goal of this section is to prove that $d_1$ defines a metric on $\Ec^1(X,\theta)$.
The following properties follow directly from the definition.
\begin{lemma}\label{lem: basic properties} Let $u,v\in \Ec^1(X,\theta)$. Then the following hold:\\
	(i) If $u\leq v$ then $d_1(u,v) = \AM(v)- \AM(u)$.\\
	(ii) If $u\leq v\leq w$ then $d_1(u,v)+d_1(v,w)= d_1(u,w)$. \\	
	(iii)(Pythagorean formula) $d_1(u,v) =d_1(u,P(u,v))+d_1(v,P(u,v))$. 
\end{lemma}
\begin{proof}
The first statement is straightforward from the definition since $P(u,v)=u$ if $u\leq v$. The second statement easily follows from $(i)$. The last statement follows from the definition of $d_1$.
\end{proof}

The following formula whose proof builds on ideas from \cite{LN15} will be crucial in the sequel.  

\begin{prop}
	\label{prop: Darvas formula 1}
	Let $u,v$ be $\theta$-psh functions with minimal singularity type. For $t\in [0,1]$ define $\varphi_t:= P((1-t)u+tv, v)$.  Then 
	\[
	\frac{d}{dt} \AM(\varphi_t) = \int_X (v-\min(u,v)) \theta_{\varphi_t}^n, \ \forall t\in [0,1]. 
	\]
\end{prop}

\begin{proof}
	We will only prove the formula for the right derivative as the same argument can be applied to treat the left derivative. Fix $t\in [0,1]$ and $s \in \Bbb R$ small such that $s  + t \in [0,1]$. For convenience we set $f_t(x):= \min((1-t)u(x)+tv(x),v(x)), \ x \in X, \ t\in [0,1]$. It follows from  \eqref{eq: env_measure_concentrate} that $\theta_{\varphi_t}^n$ is supported on the set $\{\varphi_t=f_t\}$.  By concavity of the Monge-Amp\`ere energy $\AM$ (Theorem \ref{thm: basic I energy}$(ii)$) we have that 
	\begin{eqnarray}\label{eq: firstineq}
		 \AM(\varphi_{t+s}) - \AM(\varphi_t) &\leq & \int_X (\varphi_{t+s}-\varphi_t) \theta_{\varphi_t}^n = \int_X (\varphi_{t+s}-f_t) \theta_{\varphi_t}^n \nonumber \\
	&\leq &  \int_X (f_{t+s}-f_t) \theta_{\varphi_t}^n=s \int_X  (v-\min(u,v)) \theta_{\varphi_t}^n,
	\end{eqnarray}
where in the last inequality we used that $f_{t+s}-f_t = s (v-\min(u,v))$. We use the same argument to prove the following inequality: 
	\begin{eqnarray}\label{eq: secineq}
		 \AM(\varphi_{t+s}) - \AM(\varphi_t) &\geq & \int_X (\varphi_{t+s}-\varphi_t) \theta_{\varphi_{t+s}}^n = \int_X (f_{t+s}-\varphi_t) \theta_{\varphi_{t+s}}^n \nonumber\\
	&\geq &  \int_X (f_{t+s}-f_t) \theta_{\varphi_{t+s}}^n = s  \int_X (v-\min(u,v)) \theta_{\varphi_{t+s}}^n.
	\end{eqnarray}
To continue, we notice that there exists $C>0$ such that $V_\theta - C \leq \varphi_t \leq V_\theta + C, \  t \in [0,1]$, in particular all these potentials have minimal singularity type. In addition to this, $\varphi_{t + s} \to \varphi_t$ uniformly, as $s \to 0$.

Moreover, since $v-\min(u,v)$ is a bounded quasi continuous function on $X$, the last statement of \cite[Theorem 2.3]{DDL17} is applicable to \eqref{eq: firstineq} and \eqref{eq: secineq} as $s \to 0$, to conclude that 
	\[
	\lim_{s\to 0} \frac{ \AM(\varphi_{t+s})- \AM(\varphi_t)}{s} =  \int_X (v-\min(u,v)) \theta_{\varphi_{t}}^n. 
	\]
	This completes the proof.
\end{proof}

\begin{coro}
	\label{cor: Darvas formula 2}
	Let $u,v,\varphi_t$ as in Proposition \ref{prop: Darvas formula 1}. Then 
	\[
	\AM(v)-\AM(P(u,v)) = \int_0^1 \int_X (v-\min(u,v)) \theta_{\varphi_t}^n dt. 
	\]
\end{coro}
%\begin{proof}
%	The function $t\mapsto I(P(f_t))$ is of class $\mathcal{C}^1$ in $[0,1]$ thanks to Proposition \ref{prop: Darvas formula 1}  and \cite[Lemma 4.1]{DDL17}. 
%\end{proof}
As a consequence we obtain the following result, which is an original result in the particular case of K\"ahler structures as well. 
\begin{prop}
	\label{prop: Darvas formula 3}
	If  $u,v\in \Ec^1(X,\theta)$ then $d_1(\max(u,v),u)\geq d_1(v,P(u,v))$. 
\end{prop}
\begin{proof}
Set $\varphi=\max(u,v)$, $\psi= P(u,v)$. Observe that since $v\geq \psi$ and $\varphi \geq u$, it suffices to show that $\AM(v)-\AM(\psi)\leq \AM(\varphi)-\AM(u)$. 

Recall  that for any $\chi \in \textup{PSH}(X,\theta)$ the sequence of potentials with minimal singularity type $\chi_k:= \max(\chi,V_\theta -k)$ decreases to $\chi$. 
	Consequently, using approximation (Proposition \ref{prop: BEGZ convergence I and I1}), we can assume that both $u$ and $v$ (hence also $\varphi$ and $\psi$) have minimal singularity type. Using the formula for the derivative of  $t\mapsto I((1-t)u + t\varphi)$ \cite[eq (2.2)]{BBGZ13} (or Corollary 3.3 with the choice $v:=\varphi=\max(u,v)$ in which case $P(u,v)=\min(u,v)= u$) we can write
	\[
	\AM(\varphi)-\AM(u) = \int_0^1 \int_X (\varphi-u) \theta_{(1-t)u +t\varphi}^n \,dt. 
	\] 	
Set $w_t:=(1-t)u +tv$, for $t\in [0,1]$. 
Using the trivial identity $\varphi-u=\id_{\{v>u\}}(v-u)$ and Lemma \ref{lem: plurifine_prop} we can write
\begin{equation*}
\AM(\varphi)- \AM(u) = \int_0^1 \int_{\{v>u\}} (v-u) \theta_{w_t}^n\, dt.
\end{equation*}
On the other hand, it follows from \eqref{eq: env_measure_concentrate} that 
\[
\theta_{P(w_t,v)}^n \leq \id_{\{w_t\leq v\}} \theta_{w_t}^n +  \id_{\{w_t\geq v\}} \theta_{v}^n. 
\]
Using this, Corollary \ref{cor: Darvas formula 2} and the fact that $\{w_t<  v\} = \{u < v\}$, for $t\in (0,1)$, we get
\begin{equation*}
\AM(v) -\AM(\psi)= \int_0^1 \int_X  (v-\min(u,v)) \theta_{P(w_t, v)}^n \, dt\leq \int_0^1 \int_{\{u<v\}}  (v-u) \theta_{w_t}^n\, dt,
\end{equation*}
hence the conclusion. 
\end{proof}

\begin{coro}
\label{cor: Darvas formula 3}
If $u,v,\varphi\in \Ec^1(X,\theta)$ then $d_1(u,v)\geq d_1(P(u,\varphi),P(v,\varphi))$. 
\end{coro}
\begin{proof}
	We first assume that $v\leq u$. It follows that $v\leq \max(v,P(u,\varphi))\leq u$, hence by Lemma \ref{lem: basic properties}$(iii)$ and Proposition \ref{prop: Darvas formula 3} we have
\[
		d_1(v,u) \geq d_1(v, \max(v,P(u,\varphi))) \geq  d_1(P(u,\varphi),P(P(u,\varphi),v)) = d_1(P(u,\varphi),P(v,\varphi)). 
\]
Observe that the last identity follows from the fact and $P(P(u,\varphi),v)=P(u,\varphi,v)$ and $P(u,\varphi,v)= P(\varphi, v)$ since $v\leq u$.
Now, we remove the assumption $u\geq v$. Since $\min(u,v)\geq P(u,v)$ we can use the first step to write
\[
d_1(u,P(u,v)) \geq d_1(P(u,\varphi),P(u,v,\varphi)); \quad d_1(v,P(u,v))\geq d_1(P(v,\varphi),P(u,v,\varphi)). 
\]
To finish the proof, it suffices to use Lemma \ref{lem: basic properties}(iii) and to note that $P(P(u,\varphi),P(v,\varphi))=P(u,v,\varphi)$. 
\end{proof}

\medskip

\begin{theorem}
	%\label{thm: d is a metric}
	$d_1$ is a distance on $\Ec^1(X,\theta)$.
\end{theorem}
\begin{proof}
The quantity $d_1$ is non-negative, symmetric and finite by definition. 
Next we show that $d_1$ is non degenerate. Suppose $d_1(u,v)=0$. Lemma \ref{lem: basic properties}(iii) implies that $d_1(u,P(u,v))=d_1(v,P(u,v))=0$. Moreover, Theorem \ref{thm: basic I energy}(iii) gives that $P(u,v) \geq u$ a.e. with respect to $\theta_{P(u,v)}^n$. By the domination principle (Proposition \ref{prop: domination}) we obtain that $P(u,v) \geq u$, hence trivially $u=P(u,v)$. By symmetry $v=P(u,v)$, implying that $u=v$. 

It remains to check that $d_1$ satisfies the triangle inequality: for $u,v,\varphi\in \Ec^1(X,\theta)$ we want to prove that
	\[
	d_1(u,v)\leq d_1(u,\varphi) +d_1(v,\varphi). 
	\] 
	Using the definition of $d_1$  (see \eqref{eq: d_1_eq}) this amounts to showing that 
	\[
	 \AM(P(\varphi,u)) -\AM(P(u,v))\leq \AM(\varphi) -\AM(P(\varphi,v)). 
	\]
	But this follows from Corollary \ref{cor: Darvas formula 3}, as we have the following sequence of inequalities: 
    \begin{eqnarray*}
    \AM(\varphi) - \AM(P(\varphi,v)) &=& d_1(\varphi, P(\varphi, v)) \\
    &\geq & d_1( P(\varphi, u), P(P(\varphi, v), u))=\AM(P(\varphi, u)) - \AM( P(\varphi, v, u))\\
    &\geq &  \AM(P(\varphi,u))-\AM(P(u,v)),
    \end{eqnarray*}
	where in the last line we have used the montonicity of $\AM$ (Theorem \ref{thm: basic I energy}). 
\end{proof}

\subsection{Completeness of $(\mathcal{E}^1,d_1)$}

\noindent We first establish the following key comparison between $I_1$ and $d_1$, extending \cite[Theorem 3]{Dar15} from the K\"ahler case. This result allows to interpret $d_1$-convergence using analytic means. 
\begin{theorem}\label{thm: comparison d1 and I1}
Given $u,v \in \mathcal E^1$ the following estimates hold:
	\[
	\frac{1}{3 \cdot 2^{n+2}(n+1)}I_1(u,v) \leq d_1(u,v) \leq I_1(u,v).\]
\end{theorem}

\begin{proof}
It follows from  Lemma \ref{lem: basic properties}  that $d_1(u, v) = d_1(u, P(u,v))+ d_1(v, P(u,v))$. Since the Monge-Amp\`ere energy is concave along affine curve (Theorem \ref{thm: basic I energy}(iii)),
\begin{eqnarray*}
d_1(u, P(u,v))&=& \AM(u)- \AM(P(u,v)) \leq  \int_X (u-P(u,v)) \theta_{P(u,v)}^n\\
&\leq &  \int_{\{v=P(u,v)\}} (u-v) \theta_v^n \leq \int_X |u-v| \theta_v^n .
\end{eqnarray*}
Similarly we get $d_1(v, P(u,v))\leq \int_X |u-v| \theta_u^n $. Putting these two inequalities together we get $d_1(u,v)\leq I_1(u,v)$.

Next we establish the lower bound for $d_1$. 
By the next  lemma and the Pythagorean formula we can start writing
\begin{flalign*}
\frac{3(n+1)}{2}d_1(u,v) &\geq d_1\Big(u,\frac{u + v}{2}\Big) \geq d_1\Big(u,P\Big(u,\frac{u + v}{2}\Big)\Big)\\
&\geq \int_X \Big(u - P\Big(u,\frac{u + v}{2}\Big)\Big) \theta_{u}^n.
\end{flalign*}
By a similar reasoning as above, and the fact that $2^n \theta^n_{(u + v)/2} \geq \theta^n_{u}$ we can write:
\begin{flalign*}
\frac{3(n+1)}{2}d_1(u,v) &\geq d_1\Big(u,\frac{u + v}{2}\Big) \geq d_1\Big(\frac{u+v}{2},P\Big(u,\frac{u + v}{2}\Big)\Big)\\
&\geq \int_X \Big(\frac{u+v}{2} - P\Big(u,\frac{u + v}{2}\Big)\Big) \theta_{(u + v)/2}^n\\
&\geq \frac{1}{2^n} \int_X \Big(\frac{u+v}{2} - P\Big(u,\frac{u + v}{2}\Big)\Big) \theta^n_{u}.
\end{flalign*}
Adding the last two estimates we obtain
\begin{flalign*}
 3 \cdot 2^{n}(n+1) d_1(u,v) &\geq  \int_X \Big(\Big(u - P\Big(u,\frac{u + v}{2}\Big)\Big)+\Big(\frac{u+v}{2} - P\Big(u,\frac{u + v}{2}\Big)\Big) \Big) \theta^n_{u}\\
&\geq \frac{1}{2}\int_X |u - v|  \theta^n_{u}.
\end{flalign*}
By symmetry we also have $3 \cdot 2^{n+1} (n+1) d_1(u,v) \geq \int _X |u - v|\theta_{v}^n$, and adding these last two estimates together the lower bound for $d_1$ is established.
\end{proof}

According to this last result, $I_1$ satisfies the quasi-triangle inequality. For a proof of this fact using only the pluripotential comparison principle we refer to \cite{GLZ17}.

\begin{lemma} \label{lem: halfwayest} Suppose $u,v \in \mathcal E^1$. Then the following holds:
$$d_1\Big(u,\frac{u+v}{2}\Big) \leq \frac{3(n+1)}{2}d_1(u,v).$$
\end{lemma}

\begin{proof} Using Lemma \ref{lem: basic properties} and Theorem \ref{thm: basic I energy} multiple times we deduce the following estimates:
\begin{flalign*}
d_1\Big(u,\frac{u + v}{2}\Big) &= d_1\Big(u, P\Big(u,\frac{u + v}{2}\Big)\Big) + d_1\Big(\frac{u + v}{2},P\Big(u,\frac{u + v}{2}\Big)\Big)\\
&\leq d_1(u, P(u,v)) + d_1\Big(\frac{u+v}{2},P(u,v)\Big)\\
&\leq  \int_X (u - P(u,v))\theta^n_{P(u,v)} + \int_X\Big(\frac{u+v}{2} - P(u,v)\Big)\theta^n_{P(u,v)}\\
&\leq \frac{3}{2}\int_X (u - P(u,v)) \theta^n_{P(u,v)} + \frac{1}{2}\int_X(v - P(u,v))\theta^n_{P(u,v)}\\
& \leq \frac{3(n+1)}{2}d_1(u, P(u,v)) + \frac{n+1}{2}d_1(v,P(u,v))\\
&\leq \frac{3(n+1)}{2}d_1(u,v),
\end{flalign*}
where in the second line we have additionally used that $P(u,v) \leq P(u,(u + v)/2)$.
\end{proof}

\begin{lemma}\label{lem: control sup by d1}
There exists $A, B\geq 1$ such that for any $\varphi\in \mathcal{E}^1(X, \theta)$
$$-d_1(V_{\theta}, \varphi)\leq \sup_X \varphi \leq A d_1(V_{\theta}, \varphi)+B.$$
\end{lemma}
\begin{proof}
If $\sup_X \varphi \leq 0$, then the right-hand side inequality is trivial, while
$$
-d_1(V_{\theta},\varphi)=\AM(\varphi) \leq \sup_X (\varphi-V_\theta)=\sup_X \varphi.
$$
We therefore assume that $\sup_X \varphi \geq 0$. In this case the left-hand inequality is trivial. It follows from Lemma \ref{lem: Berman Vtheta} that  $\theta_{V_{\theta}}^n\leq C dV$, for a uniform constant $C>0$.  Let $b>0$ be a constant so that $\theta \leq b\omega$. Then all $\theta$-psh functions are $b\omega$-psh. By compactness property of the set of normalized $b\omega$-psh functions (see \cite[Proposition 2.7]{GZ05}) we have
\[
\int_X |\varphi-\sup_X \varphi -V_{\theta}| \theta_{V_{\theta}}^n \leq C', 
\]
where $C'>0$ is a uniform constant. Using Theorem \ref{thm: comparison d1 and I1} the result then follows in the following manner: 
\begin{flalign*}
d_1(V_{\theta},\varphi) & \geq  D I_1(V_{\theta},\varphi)  \geq  D\int_X |\varphi-V_{\theta}| \theta_{V_{\theta}}^n\\
&\geq  D \sup_X \varphi - D \int_X |\varphi-\sup_X \varphi -V_{\theta}| \theta_{V_{\theta}}^n \geq D \sup_X \varphi -DC'.  	
\end{flalign*}

\end{proof}

With the comparison between $d_1$ and $I_1$ (Theorem \ref{thm: comparison d1 and I1}) and Lemma \ref{lem: control sup by d1} in our hands, we follow the ideas  from the proof of \cite[Theorem 9.2]{Dar14} and the convergence results in \cite{BEGZ10} to prove the next completeness theorem. 
\begin{theorem}%\label{thm: E1 complete }
The space $\left(\mathcal{E}^1(X, \theta), d_1\right)$ is complete.
\end{theorem}
 
\begin{proof}
Given $\{\varphi_j\}_j \subset \mathcal{E}^1$  a Cauchy sequence for $d_1$ we want to extract a convergent subsequence. We can assume that 
$$ 
d_1(\varphi_j,\varphi_{j+1}) \leq 2^{-j}, j \geq 1. 
$$
As in the proof of \cite[Theorem 9.2]{Dar14} we introduce the following sequences 
\[
\psi_{j,k}:=P(\varphi_j,  \varphi_{j+1}, \dots,\varphi_{k}), \ j\in \mathbb{N}, k\geq j. 
\]
Observe that, for $k\geq j+1$, $\psi_{j,k}=P(\varphi_{j},  \psi_{j+1,k})$ and hence it follows from   Lemma \ref{lem: basic properties}(iii) that
$$
d_1( \varphi_j, \psi_{j,k}) \leq d_1( \varphi_j,  \psi_{j+1,k})
\leq d_1(\varphi_j, \varphi_{j+1}) +d_1( \varphi_{j+1}, \psi_{j+1,k})
\leq \frac{1}{2^j} + d_1( \varphi_{j+1}, \psi_{j+1,k}).$$
Repeating this argument several times we arrive at 
\begin{equation}
	\label{eq: completness 1}
	d_1(\varphi_j,\psi_{j,k}) \leq 2^{-j+1}, \ \forall k\geq j+1.
\end{equation}
Using the triangle inequality for $d_1$ and the above we see that
\begin{eqnarray*}
d_1(V_{\theta},\psi_{j,k}) &\leq & d_1(V_{\theta},\varphi_j) + d_1(\varphi_j,\psi_{j,k}) \leq  d_1(V_{\theta}, \varphi_1) + 2 + 2^{-j+1}
\end{eqnarray*}
is uniformly bounded. 
It follows from Theorem \ref{thm: comparison d1 and I1} that 
$I_1(V_\theta,\psi_{j,k})$ is uniformly bounded hence
 $\psi_j:=\lim_{k} \psi_{j,k} $ belongs to $\mathcal{E}^1(X,\theta)$ (\cite[Proposition 2.19]{BEGZ10}). Moreover Proposition \ref{prop: BEGZ convergence I and I1} gives that $d_1(\psi_{j,k},\psi_j) \to 0$. From 
	\eqref{eq: completness 1} we obtain that $d_1(\varphi_j,\psi_j) \leq 2^{-j+1}$, hence we only need to show that the $d_1$-limit of the increasing sequence $\{\psi_j\}_j  \subset \mathcal{E}^1$ is in $\mathcal E^1$.
  
Lemma \ref{lem: control sup by d1} gives that $\sup_X \psi_j$ is uniformly bounded, hence $\psi:= \lim_j \psi_j \in \textup{PSH}(X,\theta)$.	Now $\psi_j$ increases a.e. towards $\psi$, hence by \cite[Proposition 2.14]{BEGZ10} $\psi \in \mathcal{E}^1(X,\theta)$. By Proposition \ref{prop: BEGZ convergence I and I1} we also have that $I_1(\psi_j,\psi)\to 0$. It follows therefore from Theorem \ref{thm: comparison d1 and I1} that
$d_1(\psi_j,\psi) \rightarrow 0$.
\end{proof}

At the end of this section we show that $d_1$-convergence in fact implies $L^1$ convergence of the potentials with respect to any fixed measure $\theta_\psi^n, \ \psi \in \mathcal E^1$, this being the analog of \cite[Theorem 5(ii)]{Dar15}:

\begin{theorem} \label{theorem: int_d1_est} For any $C>0$ there exists a continuous increasing function $g_C:\Bbb R^+ \to \Bbb R^+$ with $g_C(0)=0$ such that
\begin{equation}\label{eq: int_d_1_est} 
\int_X |u-v|\theta_{\psi}^n \leq g_C(d_1(u,v)),
\end{equation}
for $u,v,\psi \in \mathcal E_1$  satisfying $d_1(V_{\theta},u),d_1(V_{\theta},v),d_1(V_{\theta},\psi) \leq C$. 
\end{theorem}

\begin{proof} By the triangle inequality for $d_1$ it follows that $d_1(u,v)\leq 2C$, hence $I_1(u,v)\leq C_1:=3\cdot 2^{n+3}(n+1)C$, by Theorem \ref{thm: comparison d1 and I1}.  It then follows from \eqref{eq: plurifine Pythagore max} that $I_1(\max(u,v),u) \leq C_1$, hence again Theorem \ref{thm: comparison d1 and I1} yields $d_1(\max(u,v),u)\leq C_1$. The triangle inequality for $d_1$ and  Theorem \ref{thm: comparison d1 and I1} then give $I_1(\max(u,v), V_{\theta})\leq C_2$, where $C_2$ depends on $n,C_1$. Therefore,    $I_1(V_{\theta},u),I_1(V_{\theta},v),I_1(V_{\theta},\psi),I_1(\max(u,v), V_{\theta})$ are uniformly bounded by a constant  $A>0$ depending on $C$. Consequently \eqref{eq: I_script_est2} and \eqref{eq: plurifine Pythagore max} give that
\begin{equation}\label{eq: interm_est}
\int_X (\max(u,v)-v)\theta_{\psi}^n \leq f_A(I_1(\max(u,v),v)) + \int_X |u-v|\theta_u^n \leq f_A(I_1(u,v)) + I_1(u,v).
\end{equation}
Similarly, 
\begin{equation}\label{eq: interm_est 2}
\int_X (\max(u,v)-u)\theta_{\psi}^n \leq f_A(I_1(u,v)) + I_1(u,v).
\end{equation}
Since $|u - v|= (\max(u,v) - v) + (\max(u,v)-u)$, from Theorem \ref{thm: comparison d1 and I1},   \eqref{eq: interm_est} and \eqref{eq: interm_est 2} we obtain \eqref{eq: int_d_1_est}.
\end{proof}

\subsection{Geodesic segments in $(\mathcal{E}^1,d_1)$}

In this subsection we show that the weak geodesics introduced in Section 2.3 give rise to metric geodesics with respect to the $d_1$ metric geometry.
We first establish the following elementary result. 
\begin{lemma}\label{lem: stability of geodesic segments}
	Assume that $\varphi,\psi \in \mathcal{E}^1$. Let $\varphi^j, \psi^j$ be sequences of $\theta$-psh functions with minimal singularity type decreasing to $\varphi$ and $\psi$ respectively. For each $j$, let $t \mapsto \varphi_t^j$ be the Mabuchi geodesic segment connecting $\varphi^j$ and $\psi^j$. Then $\varphi_t^j$ decreases  to $\varphi_t, t \in [0,1]$, the Mabuchi geodesic segment connecting $\varphi$ and $\psi$. 
\end{lemma}

\begin{proof}
	Any candidate defining the geodesic $\varphi_t$ is also a candidate defining $\varphi_t^j$ since $\varphi\leq \varphi^j$ and $\psi \leq \psi^j$. Hence by the definition of weak geodesics in Section \ref{section ray} it follows that $\varphi_t\leq \varphi_t^j, \ t \in [0,1]$ for all $j$. Moreover the decreasing limit of the $t \to \varphi_t^j$, as $j\to +\infty$,   is a candidate in the definition of $t \to \varphi_t$, hence the conclusion. 
\end{proof}

As remarked in the preliminaries, if $\varphi,\psi \in \mathcal{E}^1$ then $P(\varphi,\psi) \in \mathcal{E}^1$ as proved in \cite[Theorem 2.10]{DDL16} (in the K\"ahler case this was addressed in \cite[Corollary 3.5]{Dar14}). Since the constant geodesic $t \to P(\varphi,\psi)$ is a candidate for $t \to \varphi_t$, the weak geodesic connecting $\varphi,\psi$, it follows that $P(\varphi,\psi) \leq \varphi_t \in \mathcal E^1$ \cite[Proposition 2.14]{BEGZ10}, hence  we may call $t \to \varphi_t$ the \emph{finite energy geodesic} connecting $\varphi,\psi$. Next we show that not only does $t \to \varphi_t$ stay inside $\mathcal E^1$, but it also has special geometric properties inside this space.

\begin{prop}
 	Let $[0,1]\ni t \mapsto \varphi_t \in \mathcal E^1$ be the finite energy geodesic connecting $\varphi_0,\varphi_1$ in $\Ec^1$. Then $t \to \varphi_t$ is a geodesic in the metric space $(\Ec^1(X, \theta),d_1)$, i.e., for any $t, s \in [0,1]$ we have
\begin{equation*}%\label{eq: d_1_geod_property}
d_1(\varphi_{t},\varphi_s) = |t-s|d_1(\varphi_0,\varphi_1),\quad  \forall t,s\in [0,1].
\end{equation*}
\end{prop}

\begin{proof}
Let $\varphi_0^j, \varphi_1^j $ be sequences of $\theta$-psh functions with minimal singularity type decreasing to $\varphi_0, \varphi_1$, respectively. Combining Lemma \ref{lem: stability of geodesic segments}, Theorem \ref{thm: comparison d1 and I1} and Proposition \ref{prop: BEGZ convergence I and I1} we obtain that $d_1(\varphi_t^j, \varphi_s^j)\rightarrow d_1(\varphi_t, \varphi_s)$ as $j$ goes to $+\infty$.
Consequently, we can then assume that $\varphi_0, \varphi_1$ have minimal singularity type and that $t \to \varphi_t$ is a weak geodesic segment with minimal singularity type. 

Since for each $t\in (0,1]$ the curve $[0,t] \ni \ell \mapsto \varphi_{\ell}$ is a weak geodesic segment connecting $\varphi_0$ and $\varphi_t$, it suffices to treat  the case when $s=0, t\in (0,1]$. It follows from \cite[Theorem 3.12]{DDL16} that $\AM$ is linear along $\varphi_t$, i.e. $\AM(\varphi_t)= t\AM(\varphi_1)+(1-t)\AM(\varphi_0)$. Hence by definition of $d_1$ we have 
	\begin{equation}\label{constant speed}
	\frac{1}{2}(d_1(\varphi_t,\varphi_0) - td_1(\varphi_0,\varphi_1)) = t\AM(P(\varphi_0,\varphi_1)) + (1-t) \AM(\varphi_0) - \AM(P(\varphi_t,\varphi_0)). 
	\end{equation}
	Let $[0,1]\ni t\mapsto \psi_t$ be the weak geodesic segment connecting $\varphi_0$ and $P(\varphi_0,\varphi_1)$. By the comparison principle (Proposition \ref{prop: DDL_comp_princ_geod}) we have that $\psi_t\leq \varphi_t$ and $\psi_t\leq\varphi_0$ for any $t\in [0,1]$, hence $\psi_t\leq P(\varphi_t,\varphi_0), \ t \in [0,1]$.  The fact that $t\mapsto \AM(\psi_t)$ is affine together with the monotonicity of $I$ give
	\begin{equation}\label{eq 1}
	t\AM(P(\varphi_0,\varphi_1)) + (1-t) \AM(\varphi_0) - \AM(P(\varphi_t,\varphi_0))\leq t\AM(P(\varphi_0,\varphi_1)) + (1-t) \AM(\varphi_0) - \AM(\psi_t)=0. 
	\end{equation}
Combining \eqref{constant speed} and \eqref{eq 1} we get $d_1(\varphi_t,\varphi_0)\leq t d_1(\varphi_1,\varphi_0)$. By symmetry it follows that $d_1(\varphi_t,\varphi_1) \leq (1-t) d_1(\varphi_0,\varphi_1)$. These two inequalities combined with the triangle inequality imply
$$t d_1(\varphi_1,\varphi_0)\geq d_1(\varphi_t,\varphi_0) \geq d_1(\varphi_0,\varphi_1)- d_1(\varphi_t,\varphi_1)\geq t d_1(\varphi_0,\varphi_1),$$
hence $d_1(\varphi_t,\varphi_0)=t d_1(\varphi_1,\varphi_0)$.
\end{proof}

\section{Construction of weak geodesic rays}\label{section: rays}

In the K\"ahler case Ross and Witt Nystr\"om described a very general method to construct weak geodesic rays with bounded potentials \cite{RWN}. In this section we show that their construction generalizes to the big case to construct weak geodesic rays with potentials of minimal singularity type. We fix from now on a potential $\phi \in \PSH(X,\theta)$ with minimal singularity type.   
\subsection{From test curves to subgeodesic rays and back}
 We first start with a few definitions.

\begin{definition}\ref{def: test curve}
We say that a (weak) subgeodesic ray $t\mapsto h_t$ with minimal singularity type is $t$-Lipschitz if there exists $L>0$ such that 
$$
h_t(x) \leq h_s(x) +  L |t-s|, \ \forall t,s\in \mathbb{R}^+, \forall x\in X. 
$$
\end{definition}
Examples of $t$-Lipschitz  subgeodesics are $t\rightarrow \max(\psi,\phi-t)$, where $\psi \in  \PSH(X,\theta)$ with  minimal singularity type. Also,  note that weak geodesic rays with minimal singularity type $t \to \psi_t$ are automatically $t$-Lipschitz with $L=\max\{|m_{\psi_t}|,|M_{\psi_t}|\}$ (see \eqref{eq: m_M_normalization}).

Following the terminology of Ross and Witt Nystr\"om \cite{RWN} we introduce test curves:

\begin{definition}\label{def: test curve}
A map $\mathbb{R}\ni \tau \rightarrow \psi_\tau\in \PSH(X, \theta)$ is a \emph{test curve} if
\begin{itemize}
\item[(i)] $\tau\rightarrow \psi_\tau(x)$ is concave for any $x\in X$,
\item [(ii)] there exists $C_{\psi}>0$ such that $\psi_\tau$ is equal to some potential with minimal singularity type $\phi\in \psh(X, \theta)$ for $\tau < -C_\psi$, and $\psi_\tau\equiv -\infty$ if $\tau >C_\psi$. 
\end{itemize}
\end{definition}

In the case when $\theta$ is K\"ahler, treated in \cite{RWN}, the additional assumption that each $\psi_\tau \in \psh(X,\theta)$ has small unbounded locus was also included in the above definition. Ross and Witt Nystr\"om later observed that this assumption is not necessary (see the proof of \cite[Theorem 2.9]{Dar13}) and we work with this more general definition here as well.

We recall the \emph{Legendre transform}, adjusted to our special case of interest. Given a convex function  $[0, +\infty)\ni t\rightarrow f(t) \in \Bbb R$, its Legendre transform is defined as \[\hat f(\tau):= \inf_{t \geq 0} (f(t)-t\tau), \ \tau\in \mathbb{R}.
 \]
The \emph{(inverse) Legendre transform} of a decreasing concave function $\mathbb{R}\ni\tau\rightarrow g(\tau) \in \mathbb{R}\cup \{-\infty\}$ is
\[
\check{g}(t):=\sup_{\tau \in \Bbb R} (g(\tau)+t\tau), \ t \geq 0.
\]
We point out that there is a sign difference in our choice of Legendre transform compared to the literature, however this particular choice will be more suited in the context of our investigations. 

As it is well known, for every $\tau \in \mathbb{R}$ we have that $\hat{\check{g}}(\tau) \geq g(\tau)$ with equality if and only if $g$ is additionally upper semicontinuous at $\tau$. Similarly, $\check{\hat{f}}(t) \leq f(t)$ for all $t\geq 0$ with equality if and only if $f$ is lower semicontinuous at $t$.   
We will refer to these identities as the involution property of the Legendre transform. For a detailed treatment of Legendre transforms we refer to \cite[Chapter 26]{Rock}.

Starting with a test curve $\tau \to \psi_\tau$, our goal will be to construct a geodesic/subgeodesic ray by taking the $\tau$ inverse Legendre transform of $\tau \to \psi_\tau$. As shown below,  the resulting curve $t \to \check \psi_t$ is a subgeodesic, and under additional conditions it will be a weak geodesic. 

Before we detail our constructions, let us first address one annoying technical issue. Let $\tau \to \psi_\tau$ be a test curve, and $x \in X$. By Definition \ref{def: test curve}, the concave function $\tau \to \psi_\tau(x)$ may not be $\tau$-usc (usc in the the $\tau$ direction), hence the involution property may not hold for it, i.e., $\hat {\check \psi}_\tau \neq \psi_\tau$. We address this with the next simple lemma, that points out that by changing at most one K\"ahler potential along the curve $\tau \to \psi_\tau$, we get a new $\tau$-usc test curve, whose (inverse) Legendre transform coincides with the one of $\tau \to \psi_\tau$. Consequently, \emph{there is no loss of generality in considering $\tau$-usc test curves in our constructions below.}

To start, for a test curve $\tau \to \psi_\tau$ we introduce the following two constants:
$$\tau_\psi^+ := \inf\{\tau \in \Bbb R \ | \ \psi \equiv -\infty\},$$
$$\tau_\psi^- := \sup\{\tau \in \Bbb R \ | \ \psi \equiv \phi\}.$$

\begin{lemma}\label{lem: uscify}
	Let $\tau \to \psi_{\tau}$ be a test curve and $x \in X$. Then $(-\infty,\tau_\psi^+) \ni \tau \to \psi_\tau(x)$ is $\tau$-usc. Additionally, the test curve $\tau \to \overline\psi_\tau$ defined below is $\tau$-usc, and $\check{\overline \psi}_t = \check\psi_t$ for all $t \geq 0$.
$$
\overline \psi_\tau=
\begin{cases}
\psi_\tau \  \ \ \ \ \ \ \ \ \ \  \textup{ if } \ \tau < \tau_\psi^+,\\
\lim_{s \nearrow \tau_\psi^+} \psi_s \ \textup{ if } \ \tau = \tau_\psi^+,\\
-\infty  \  \  \  \ \  \ \ \ \ \textup{ if } \ \tau > \tau_\psi^+.
\end{cases}
$$
\end{lemma}

\begin{proof}
Suppose $s \in (-\infty, \tau^+_\psi)$. By concavity and the fact that  $\psi_{\tau} = \phi$ for $\tau \leq -C_{\psi}$ it follows that $\tau \to \psi_{\tau}(x)$ is decreasing. Let $u_s$ be the decreasing limit of $\psi_{\tau}, \tau <s$, which is $\theta$-psh. To show that $\tau \to \psi_\tau(x)$ is upper semicontinuous  at $s$, it suffices to prove that $u_s=\psi_s$ everywhere on $X$. Fix $t\in (s, \tau^+_\psi)$ and $x_0\in X$ such that $\psi_{t}(x_0)>-\infty$. It follows that $(-\infty, t)\ni \tau \to \psi_{\tau}(x_0)$ is continuous, hence $u_s(x_0)=\psi_s(x_0)$. Thus $\psi_s = u_s$  almost everywhere in $X$. Since $u_s$ and $\psi_s$ are both quasi-plurisubharmonic, this implies that $\psi_s = u_s$ everywhere (see Lemma \ref{lem: aeineq_everywhere}). 

By the above, we obtain that $\tau \to \overline \psi_\tau(x_0)$ is $\tau$-usc on $\Bbb R$ for any $x_0 \in X$. Additionally, comparing with Definition \ref{def: test curve}, $\tau \to \overline \psi_\tau$ is also test curve and by the definition of the (inverse) Legendre transform we get that $\check{\overline \psi}_t = \check \psi_t$ for all $t \geq 0$. 
\end{proof}

We are ready to establishing the duality between test curves and subgeodesic rays:

\begin{prop}\label{prop: subgeodesic vs test curve}
 The  map $ \psi \to \check \psi$ gives a bijection between $\tau$-usc test curves $\tau \to \psi_\tau$ and $t$-Lipschitz subgeodesic rays $t \to h_t$, with inverse $h \to \hat h$.
\end{prop}

\begin{proof}
Assume that $\tau \to \psi_{\tau}$ is a test curve such that $\psi_{-\infty}=\phi$. We want to prove that its inverse Legendre transform
$$
 h_t:= \sup_{\tau \in \mathbb{R}} \left(\psi_\tau+t\tau \right)
$$
is a $t$-Lipschitz subgeodesic ray.  By concavity of $\tau \mapsto \psi_{\tau}$ we have that $\psi_{\tau} \leq \phi, \forall \tau$, thus $h_0=\phi$. Moreover, Proposition \ref{prop: usc} below shows that $(t,x)\mapsto h(t,x)$ is $(t,x)$-upper semicontinuous. For each $\tau \in \mathbb{R}$ the curve $t\mapsto \psi_{\tau} + t\tau$ is a subgeodesic ray. Hence, as a supremum of subgeodesic rays that is upper semicontinuous, the curve $t\mapsto h_t$ is also a subgeodesic ray.  It remains to prove that $h_t$ is uniformly Lipschitz in $t$ which is equivalent to showing that $|h_t-\phi| \leq Ct, \forall t\geq 0$, for some positive constant $C>0$. But the latter follows since $\tau \to \psi_{\tau}$ is a test curve:
$$\phi -C_{\psi} t \leq \sup_{\tau \in \mathbb{R}} \left(\psi_\tau+t\tau \right)=\sup_{\tau \in (-\infty, C_{\psi}]} \left(\psi_\tau+t\tau \right)\leq \phi+C_{\psi}t.$$
Since $\tau \to \psi_\tau$ is assumed to be $\tau$-usc, by the involution property we have that $\check{\hat{\psi}}=\psi$. 

To finish the proof, we only have to argue that $\hat{h}$ is a $\tau$-usc test curve for all $t$-Lipschitz subgeodesics $t \to h_t$. Since $t \to h_t$ is $t$-convex and $t$-continuous, it is clear that $\tau \to \hat h_\tau$ is $\tau$-concave and $\tau$-usc. On the other hand, Kiselman's minimum principle (Theorem \ref{thm: Kiselman}) implies that $\hat h_\tau \in \textup{PSH}(X,\theta)$. Since $t \to h_t$ is $t$-Lipschitz (with Lipschitz constant equal to $L$), it follows that property (ii) of test curves also holds for $\tau \to \hat h_\tau$ with $C_{\hat h}=L$. Indeed, if $\tau >L$ we have 
$$
\inf_{t\geq 0} (h_t-t\tau)\leq \inf_{t\geq 0} (\phi+t(L-\tau))=-\infty, 
$$ 
while for $\tau <-L$ we have 
$$
\inf_{t\geq 0} (h_t-t\tau)\geq \inf_{t\geq 0} (\phi-t(L+\tau))=\phi.
$$
\end{proof}

\begin{prop}\label{prop: usc}
Let $\tau\rightarrow \psi_\tau$ be a test curve. Then the function  
\[
 [0,+\infty) \times X \ni (t,x) \mapsto   \sup_{\tau\in \mathbb{R}} ({\psi}_\tau + t\tau)
\] 
is $(t,x)$-upper semicontinuous. 
\end{prop}

\begin{proof}
Set $h_t:= \sup_{\tau \in \mathbb{R}}(\psi_{\tau}+ t\tau)$. Since $\psi$ is a test curve the supremum can be taken for $\tau\in I$, where $I$ is a compact interval of $\mathbb{R}$ (however the supremum may not be attained as $\tau \to \psi_\tau(x)$ may fail to be upper semicontinuous for some $x \in X$). It follows that $h_t$ is uniformly Lipschitz in $t$, i.e. $|h_t(x)-h_s(x) | \leq C |t-s|$, for all $t,s \in \mathbb{R}^+, x\in X$. Now the upper semicontinuity of $h(t,x)$ reduces to upper semicontinuity of $x\mapsto h_t(x)$ for each $t\geq 0$ fixed. Assume that $X \ni x_j \to x\in X$, and pick $\tau_j \in I$ such that
\[
h_t(x_j) -\frac{1}{j} \leq \psi_{\tau_j}(x_j) + \tau_j t \leq h_t(x_j).
\] 
If $\tau$ is any cluster point of $\tau_j$ then, after possibly extracting a subsequence, we may assume that $\tau_j$ converges to $\tau$. Fix $\ell_1\leq -C_\psi$, $\ell_1<\ell_2<\tau$ and $\alpha_j:= \frac{\tau_j-\ell_2}{\tau_j-\ell_1}$. For $j$ big enough $\alpha_j \in (0,1)$. Note that $(1-\alpha_j)\tau_j+\alpha_j \ell_1=\ell_2$. Hence from the concavity of 
$\tau\rightarrow \psi_\tau$ we get that 
\begin{equation}\label{conc}
\psi_{\tau_j}(x_j)\leq  \frac{1}{1-\alpha_j} \psi_{\ell_2}(x_j)- \frac{\alpha_j}{1-\alpha_j}\psi_{\ell_1}(x_j).
\end{equation}
If $\phi (x_j)\rightarrow -\infty$ then $h_t(x_j)\rightarrow -\infty$. In this case it is trivial that 
$$\limsup_{j\to +\infty} h_t(x_{j}) \leq h_t(x).$$
Consequently, after possibly extracting another subsequence, we can assume that there exists $C>0$ such that $\phi (x_j)\geq -C$. Since $\ell_1 \leq -C_\psi$, this means that $\psi_{\ell_1}(x_j)\geq -C$. %Since $(1-\alpha_j)\tau_j + \alpha_j \ell_1 = \ell_2$ 
Using \eqref{conc} we obtain that
$$
\psi_{\tau_j}(x_j)+t \tau_j \leq  \frac{1}{1-\alpha_j} \left(\psi_{\ell_2}(x_j) +t\ell_2 \right)+t \left(\tau_j - \frac{1}{1-\alpha_j}\ell_2 \right) +C \frac{\alpha_j}{1-\alpha_j}.$$
Using the upper semicontinuity of $\psi_{\ell_2}$ in $x$, we can continue to write:
\begin{eqnarray*}
\limsup_{j\to +\infty} h_t(x_{j}) &\leq &  \frac{1}{1-\alpha} \left(\psi_{\ell_2}(x) +t\ell_2 \right)+t \left(\tau - \frac{1}{1-\alpha}\ell_2 \right) +C \frac{\alpha}{1-\alpha} \\
&\leq & \frac{1}{1-\alpha}  h_t(x) +t \left(\tau - \frac{1}{1-\alpha}\ell_2 \right) +C \frac{\alpha}{1-\alpha},
\end{eqnarray*}
where $\alpha= \frac{\tau-\ell_2}{\tau-\ell_1}.$
Letting $\ell_2\rightarrow \tau$ we get the conclusion.
\end{proof}

\subsection{From maximal test curves to geodesic rays and back}

In this subsection we generalize and slightly extend the construction of weak geodesic rays from \cite{RWN} to the setting of big cohomology classes.

Partially following the terminology of Ross and Witt Nystr\"om \cite{RWN}, a test curve $\tau \to \psi_\tau$ is said to be \emph{maximal} if 
$$P[\psi_\tau](\phi)=\psi_\tau \  \textup{ for all } \  \tau \in \Bbb R,$$ where $\phi = \psi_{-\infty}$, and we use notation and terminology from \cite[Section 1]{DDL17}, as elaborated in \eqref{eq: P_sing_def}.  

In the next result we describe a method to attach a maximal $\tau$-usc test curve to an arbitrary test curve $\tau \to \psi_\tau$. As we will see, taking the (inverse) Legendre transform of the former curve will give a weak geodesic ray.

\begin{prop}\label{prop: maximization} Suppose $\tau \to \psi_\tau$ is a test curve. Then $\tau \to \psi^M_\tau := \overline{P[\psi_\tau](\phi)}$ is a maximal $\tau$-usc test curve.
\end{prop}
 
\begin{proof} We first prove that $\tau \to \chi_\tau := P[\psi_\tau](\phi)$ is a test curve.  Fix $t<s<r \in \mathbb{R}$. Let $\lambda\in (0,1)$ be such that $s =\lambda t + (1-\lambda) r$. We want to prove that  $\chi_{s} \geq \lambda \chi_t + (1-\lambda) \chi_r$. Fix $C>0$. By $\tau$-concavity of $\tau \mapsto \psi_{\tau}$ we have 
$$
\psi_s + C \geq \lambda (\psi_t +C) +(1- \lambda) (\psi_{r}+C).
$$
Using this estimate  we see that  the function $\lambda P(\psi_t+C,\phi) + (1-\lambda) P(\psi_r+C, \phi)$  is $\theta$-psh and it is not greater than $\min(\psi_s+C,\phi)$.  Therefore, 
$$
P(\psi_r+C, \phi) \geq \lambda P(\psi_t +C,\phi) +(1- \lambda) P(\psi_{r}+C,\phi). 
$$
Letting $C\to +\infty$ we obtain $\chi_{s} \geq \lambda \chi_t + (1-\lambda) \chi_r$ a.e. on $X$. Lemma \ref{lem: aeineq_everywhere} then gives the desired concavity property. We clearly have $\chi_{\tau} = \phi$ for $\tau < \tau_{\psi}^-$ and $\chi_{\tau} =-\infty$ for $\tau > \tau_{\psi}^+$. Hence $\chi_{\tau}$ is a test curve. 
 
 By the definition of $\tau \to \psi^M_\tau$ we have that this curve is $\tau$-usc and that
$$\psi^M_\tau=
\begin{cases}
P[\psi_\tau](\phi) \  \ \ \ \ \ \ \ \ \ \ \ \; \, \textup{ if } \ \tau < \tau_\psi^+,\\
\lim_{s \nearrow \tau_\psi^+} P[\psi_s](\phi) \ \ \, \textup{ if } \ \tau = \tau_\psi^+,\\
-\infty  \  \  \  \ \  \ \ \ \ \ \ \ \ \ \ \ \ \ \; \, \textup{ if } \ \tau > \tau_\psi^+.
\end{cases}$$ 
To show that $\tau \to \psi^M_\tau$ is maximal we need to show that $P[\psi_\tau^M](\phi)=\psi_\tau^M$ for each $\tau \in \Bbb R$. For $\tau > \tau^+_\psi$ this is trivial.

Now we address the case $\tau < \tau^+_\psi$. Pick $s\in (\tau, \tau_\psi^+)$. By concavity in the $\tau$-variable, we have that $P[\psi_\tau](\phi) \geq  \alpha \phi + (1-\alpha) P[\psi_s](\phi)$ for some $\alpha \in (0,1)$. By the monotonicity of Monge-Amp\`ere mass  (see \cite[Theorem 1.2]{WN17} ) we then have 
$$
\int_X \theta_{P[\psi_\tau](\phi)}^n \geq \int_X \left (\theta + i\ddbar (\alpha \phi + (1-\alpha) P[\psi_{s}](\phi)) \right )^n \geq \alpha^n \int_X \theta_{\phi}^n= \alpha^n,
$$
where the last equality follows from the fact that $\phi$ has minimal singularity. By Lemma \ref{lem: mass envelope} we then have $\int_X \theta_{\psi_{\tau}}^n >0$. Consequently, it follows from Lemma \ref{lem: domination principle} below  that $P[\psi_\tau^M](\phi)=\psi_\tau^M.$

Lastly, we address the case $\tau := \tau^+_\psi$. If $s < \tau=\tau^+_\psi$, then by the above we can write
$$P[\psi_\tau^M](\phi) \leq P[\psi_s^M](\phi) = \psi_s^M.$$
Letting $s \nearrow \tau^+_\psi$, by the definition of $\tau \to \psi^M_\tau$, we obtain that $P[\psi_\tau^M](\phi) \leq \psi_\tau^M$. Since the reverse inequality is trivial, we obtain $P[\psi_\tau^M](\phi) = \psi_\tau^M$, hence the result follows.
\end{proof}

\begin{lemma}\label{lem: domination principle}
Suppose that $\phi \in \textup{PSH}(X,\theta)$ has minimal singularity, and $\chi \in \textup{PSH}(X,\theta)$ satisfies $\int_X \theta^n_\chi >0$.  Then $P[\chi](\phi)=P[P[\chi](\phi)](\phi)$.
\end{lemma}
\begin{proof} Since $\phi$ has minimal singularity type, from \cite[Theorem 3.12]{DDL17} it follows that the singularity type of $P[\chi](\phi)$ and $P[P[\chi](\phi)](\phi)$ is the same (because $P[\chi](V_\theta)=P[P[\chi](V_\theta)](V_\theta)$). 

Trivially $P[\chi](\phi) \leq  P[P[\chi](\phi)](\phi)$, however \cite[Theorem 3.8]{DDL17} implies that
$$\theta_{P[\chi](\phi)}^n  \leq \mathbbm{1}_{\{P[\chi](\phi)=\phi\}} \theta_\phi^n.$$
In particular, since $\{P[P[\chi](\phi)](\phi)=\phi\} \subset \{P[\chi](\phi)=\phi\}$, we have that $P[\chi](\phi) \geq  P[P[\chi](\phi)](\phi)$ a.e. with respect to $\theta_{P[\chi](\phi)}^n$. The domination principle \cite[Proposition 3.11]{DDL17} implies that $P[\chi](\phi)\geq P[P[\chi](\phi)](\phi)$, hence in fact $P[\chi](\phi)= P[P[\chi](\phi)](\phi)$.
\end{proof}

\begin{theorem}\label{thm: from test curve to ray} The  map $ \psi \to \check \psi$ gives a bijection between $\tau$-usc maximal test curves $\tau \to \psi_\tau$, and weak geodesic rays with minimal singularity type $t \to u_t$, with inverse map $u \to \hat u$.
\end{theorem}

\begin{proof} Let $\tau \to \psi_\tau$ be a $\tau$-usc maximal test curve. Denote by $t \to h_t$ the inverse Legendre transform of $\psi_\tau$, i.e. $h_t:= \sup_\tau (\psi_\tau+t\tau)$, $t\geq 0$. By Proposition \ref{prop: subgeodesic vs test curve}, $t \to h_t$ is a $t$-Lipschitz subgeodesic ray with minimal singularity type emanating from $\phi$. In particular $\phi-C_{\psi} t\leq h_t \leq \phi+ C_{\psi}t$, $t\geq 0$.

For each $D>0$, let $t \to w_t^D$ be the upper envelope of all subgeodesic rays lying below $\min(\phi+ C_{\psi} t, h_t+D)$. Then $t\mapsto w_t^D$ is a weak subgeodesic ray emanating from $\phi$ which is uniformly Lipschitz in $t$. It follows from Proposition \ref{prop: subgeodesic vs test curve} that the Legendre transform
$\tau \to {\hat{w}_\tau^{D}}$ is a $\tau$-usc test curve.

As $w^D_t \leq \min( \phi+C_\psi t , h_t+D)$, by the involution property of the Legendre transform  we obtain that $ \hat w^D_\tau = \inf_{t\geq 0}(w^D_t - t\tau) \leq \min(\phi, \psi_\tau + D)$. 
By the Kiselman minimum principle (Theorem \ref{thm: Kiselman}),  ${\hat{w}^D_\tau} \in \textup{PSH}(X,\theta)$,  hence  we conclude that
$${\hat{w}^D_\tau} \leq P(\phi, \psi_\tau + D), \   \tau \in \Bbb R.$$
We apply the inverse Legendre transform to this inequality, and use the involution property to conclude  that 
$$
{w}_t^D\leq \sup_{\tau \in \mathbb{R}} \left(P(\phi,\psi_\tau+D ) +t\tau \right), \ t \geq 0. 
$$
As we now argue, this inequality is in fact an equality. Indeed, by construction $\tau \mapsto P(\phi, \psi_{\tau}+D)$ is a test curve,  and by Proposition \ref{prop: subgeodesic vs test curve} and Lemma \ref{lem: uscify}  the curve 
 $$
 t\to  \sup_{\tau \in \mathbb{R}} \left(P(\phi,\psi_\tau+D ) +t\tau \right) =  \sup_{\tau \leq C_{\psi}} \left(P(\phi,\psi_\tau+D ) +t\tau \right)
 $$
 is a subgeodesic ray which is a candidate in the definition of $t \to {w}_t^D$. Thus
\begin{equation}\label{ray 0}
{w}_t^D=\sup_{\tau \in \mathbb{R}} \left(P(\phi,\psi_\tau+D ) +t\tau \right), \ t \geq 0.
\end{equation}
Now, observe that $D\rightarrow w_t^D$ is an increasing sequence and define 
$w_t:= \textup{usc}\left( \lim_{D\rightarrow +\infty} w_t^D\right)$. From Lemma \ref{lem: balyage} below it follows that $w_t$ is a weak geodesic ray emanating from $\phi$. 
By maximality of $\tau \to \psi_\tau$, $P(\phi, \psi_{\tau}+D) \leq P[\psi_{\tau}](\phi)=\psi_\tau$, hence letting $D\to \infty$  in \eqref{ray 0} we get that
\begin{equation}\label{ray 1}
w_t \leq  \sup_{\tau\in \mathbb{R}} \left( P[\psi_\tau](\phi)+t\tau \right)=\sup_{\tau \in \mathbb{R}} \left(\psi_\tau+t\tau \right).
\end{equation}
By construction $w_t\geq w_t^D$ and for each $\tau \in \mathbb{R}$ we have 
\[
w_{t}^D \geq P(\phi,\psi_{\tau} +D) + t\tau.
\]
Letting $D\to +\infty$ we arrive at 
\[
w_t \geq P[\psi_{\tau}](\phi) + t\tau=\psi_{\tau} + t\tau, \ \forall \tau \in \mathbb{R}.
\]
Taking the supremum over all $\tau\in \mathbb{R}$ we have the reverse inequality of \eqref{ray 1}.
It then follows that $t\rightarrow \sup_\tau \left(\psi_\tau +t\tau \right)$ is a weak geodesic ray, because so is $t \to w_t$.

By the involution property, it follows that $\hat {\check \psi}_\tau =\psi_\tau$ for any $\tau$-usc maximal test curve and $\check {\hat u}_t =u_t$ for any weak geodesic $t \to u_t$ with minimal singularity type.

From \eqref{eq: m_M_normalization} it follows that a weak geodesic ray $t \to u_t$ with minimal singularity is automatically $t$-Lipschitz, hence (via Proposition \ref{prop: subgeodesic vs test curve}) $\tau \to \hat u_\tau$ is a $\tau$-usc test curve. Lastly, \cite[Lemma 3.17]{DDL16} implies that $\tau \to \hat u_\tau$ is maximal.
\end{proof}

\begin{lemma}\label{lem: balyage}
	Assume that $[0,+\infty) \ni t \mapsto u_t, v_t \in \PSH(X,\theta)$ is a  geodesic and a subgeodesic ray respectively, both having potentials with minimal singularity type, both  emana\-ting from $\phi$. For each $C>0$ let $t \to w_t^C$ be the upper envelope of all subgeodesic rays lying below $t \to \min(u_t,v_t+C)$. If $t\mapsto v_t$ is $t$-Lipschitz then
	\[
	[0,+\infty) \ni t \mapsto w_t := \textup{usc}\Big( \lim_{C\to +\infty} w_t^C\Big) \in \textup{PSH}(X,\theta)
	\]
	is a weak geodesic ray with minimal singularity type emanating from $\phi$. 
\end{lemma}

\begin{proof}
We first observe that the set of subgeodesic rays lying below $\min(u_t,v_t)$ is non empty. Indeed, since $u_t$ has minimal singularity type, it follows from \eqref{eq: m_M_normalization} that $u_t \geq \phi-Dt$ for some positive constant $D$. Hence the curve $t\mapsto v_t-Dt$ is a subgeodesic ray lying below $\min(u_t,v_t)$. We deduce, in particular, that $t \to w_t$ is a subgeodesic ray emanating from $\phi$ and $w_t$ has minimal singularity  for all $t\geq 0$.

	Fix $s>0$. We prove that the curve $[0,s] \ni t \mapsto w_t^C$ is actually a geodesic segment when $C$ is large enough. Indeed, let $[0,s] \ni t \mapsto \varphi_t \in \textup{PSH}(X,\theta)$ be the geodesic segment connecting $w_0^C$ and $w_s^C$ and extend $\varphi_t$ to $[0,+\infty)$, by setting $\varphi_t=w_t^C$, for $t\geq s$. By Proposition \ref{prop: DDL_comp_princ_geod} we have that $\varphi_t \geq w_t^C$ for all $t\geq 0$. It follows from basic properties of plurisubharmonic functions (see \cite[Proposition 1.30]{GZbook}) that $t \to \varphi_t$ thus constructed is a  subgeodesic ray.
    
For $C>0$ big enough we have $u_t \leq v_t+C$ for all $t \in [0,s]$, since these functions have minimal singularity type. In particular $w_s^C \leq  u_s$. Recall that $[0,s]\ni t \mapsto u_t$ is the geodesic segment connecting $\phi$ to $u_s$.  Thus, for such $C$, the comparison principle (Proposition \ref{prop: DDL_comp_princ_geod}) gives $\varphi_t  \leq u_t$, for all $t\in [0,s]$. Since $u_t\leq \min(u_t,v_t+C)$, for $t\in [0,s]$ it then follows that $\varphi_t\leq \min(u_t,v_t+C)$, for all $t\geq 0$. Therefore $\varphi_t$ is a candidate defining $t\mapsto w_t^C$. This implies that $\varphi_t \leq w_t^C$ for all $t\geq 0$, hence $\varphi_t = w_t^C$, for all $t\geq 0$ (in the previous paragraph we proved the reverse inequality).  In particular $w_t^C$ is a geodesic segment in $[0,s]$.  

Letting $C \to +\infty$, by convergence of $\AM$ along increasing sequences (see Proposition \ref{prop: BEGZ convergence I and I1}) one sees that, $[0, s]\ni t \to \AM(w_t)$ is affine (\cite[Theorem 3.12]{DDL16}). Now, let $[0,s]\ni t\to \phi_t \in \mathcal E^1$ be the geodesic segment joining $w_0$ and $w_s$. Then $\phi_t\geq w_t$ and $\AM(w_t)=\AM(\phi_t)$ for any $t\in [0,s]$. In particular, we have that $\int_X (\phi_t - w_t) \theta_{w_t}^n =0$. Hence, the domination principle (Proposition \ref{prop: domination}) reveals that $w_t = \phi_t$ for all $t \in [0,s]$, i.e., $t \to w_t$ is a geodesic ray. 
\end{proof}

Finally let us state and prove the big version of the main analytic result of \cite{RWN}:

\begin{coro}\label{cor: from test curve to ray} Let $\tau\rightarrow \psi_\tau$ be a test curve such that $\psi_{-\infty} =\phi$.
Define $$w_t= \sup_{\tau\in \mathbb{R}} (P[\psi_\tau](\phi) + t\tau), \ \ t \geq 0.$$ Then the curve $t\rightarrow w_t$ is a weak geodesic ray, with minimal singularity type, emanating from $\phi$.
\end{coro}
\begin{proof} To start, from  the first step in the proof of Proposition \ref{prop: maximization} we know that $\tau \to \chi_\tau := P[\psi_\tau](\phi)$ is a test curve, and from Lemma \ref{lem: uscify} it follows that $\check \chi_t = \check \psi^M_t$, for every $t\geq 0$. Moreover Proposition \ref{prop: maximization} insures that $\tau\rightarrow \psi_\tau^M$ is a maximal  $\tau$-usc test curve. By Theorem \ref{thm: from test curve to ray} above, $t \to \check \psi^M_t$ is a weak geodesic ray with minimal singularity type. Hence, so is $t \to \check \chi_t$.
\end{proof}

\let\OLDthebibliography\thebibliography 
\renewcommand\thebibliography[1]{
  \OLDthebibliography{#1}
  \setlength{\parskip}{1pt}
  \setlength{\itemsep}{1pt plus 0.3ex}
}

\vspace{0.1cm}
 
 \noindent{\sc University of Maryland}\\
{\tt tdarvas@math.umd.edu}\vspace{0.1in}, \href{http://www.math.umd.edu/~tdarvas/}{http://www.math.umd.edu/\~{}tdarvas}   \newpage
\noindent{\sc IHES}\\
{\tt dinezza@ihes.fr}, \href{https://sites.google.com/site/edinezza/home}{https://sites.google.com/site/edinezza/home}\vspace{0.1in}\\
 \noindent {\sc Universit\'e Paris-Sud}\\
{\tt hoang-chinh.lu@math.u-psud.fr}, \href{https://www.math.u-psud.fr/~lu/}{https://www.math.u-psud.fr/\~{}lu}

\end{document}